\documentclass[11pt]{article}
\usepackage{verbatim,a4wide} 
\usepackage{graphicx}
\usepackage[utf8]{inputenc}
\usepackage{amsmath,amsfonts,amssymb,amscd,amsthm,txfonts,hyperref}
\usepackage{xcolor}
\newtheorem{theorem}{Theorem}[section]
\newtheorem{proposition}[theorem]{Proposition}
\newtheorem{lemma}[theorem]{Lemma}
\newtheorem{corollary}[theorem]{Corollary}
\newtheorem{remark}{Remark}[section]

\theoremstyle{definition}
\newtheorem{definition}[theorem]{Definition}

\newcommand{\R}{\mathbb{R}}

\newcommand{\N}{\mathbb{N}}
\newcommand{\Z}{\mathbb{Z}}
\newcommand{\T}{\mathbb{T}}

\begin{document}

\title{Continuity of the flow of KdV with regard to the Wasserstein metrics and application to an invariant measure}
\author{Federico Cacciafesta\footnote{SAPIENZA - Universit\'a di Roma, Dipartimento di Matematica, Piazzale A. Moro 2, I-00185 Roma, Italy}\; and Anne-Sophie de Suzzoni\footnote{Universit\'e de Cergy-Pontoise,  Cergy-Pontoise, F-95000,UMR 8088 du CNRS }}

\maketitle

\begin{abstract} In this paper, we prove the continuity of the flow of KdV on spaces of probability measures with respect to a combination of Wasserstein distances on $H^s$, $s>0$ and $L^2$. We are motivated by the existence of an invariant measure belonging to the spaces onto which these distances are defined. \end{abstract}

\tableofcontents

\section{Introduction}

The Korteweg de Vries (KdV) equation 
\begin{equation}\label{kdv}
\partial_t u + \partial_x^3 u + \frac{1}{2} \partial_x u^2 = 0
\end{equation}
is an approximation of the water waves in the case of large wavelengths and small initial data. 
There is an extensive literature on the Cauchy problem associated to KdV, from what we shall refer as a deterministic point of view, see for instance \cite{KTglo,bourgain,KPVbil,CKSTTsha,kato}, would it be on the torus, the Euclidean space or other manifolds. To consider KdV from a probabilistic point of view, we mention \cite{Binvkdv,Ohinv}.

In this paper we are interested in the continuity of the flow of this equation on the torus $\mathbb{T}$ with regard to metrics on functional probability spaces. However, we do not study here a problem of low-regularity of the initial data. Indeed, our initial datum is a measure $\mu$ on the topological $\sigma$-algebra of the Sobolev space $H^s$ for some non negative $s$, where KdV is known to be well-posed. We define the action of the flow of KdV on the measures as $\mu \mapsto \mu^t$ where $\mu^t$ is the pushforward measure of $\mu$ under the flow of KdV, $\Psi(t)$, that is, the measure defined as
$$
\mu^t (A) = \mu(\Psi(t)^{-1}A)
$$
on every measurable set $A$. This measure is well-defined as soon as $\Psi(t)$ is defined and continuous on $H^s$ (in this case, $\Psi(t)$ is $\mu$-measurable). 

We inquire about the continuity of the action of the flow with regard to the Wasserstein metrics (we refer to \cite{KSmat,Drea} for further informations on these metrics). In other terms, we compare the distance between $\mu^t$ and $\nu^t$ with the one between $\mu$ and $\nu$. The Wasserstein distance is defined in our case as
$$
W_{s,p}(\mu, \nu) = \inf_{\gamma \in \textrm{Marg}(\mu,\nu)} \left( \int \|x-y\|_{H^s}^p d\gamma(x,y) \right)^{1/p}
$$
where $\textrm{Marg}(\mu,\nu)$ is the set of measures on the topological $\sigma$-algebra of the Cartesian product $H^s\times H^s$ whose marginals are $\mu$ and $\nu$ . We assume that $p$ is more than $1$ and $s$ non negative, with the usual extension for $p = \infty$. This distance is used in transportation theory (see \cite{BKthMK} for instance) to represent the optimal cost to move a repartition $\mu$ of goods to a repartition $\nu$ when the price to transport one item from $x$ in the support of $\mu$ to $y$ in the support of $\nu$ is $\|x-y\|_{H^s}^p$. 

To be more precise, what we will use as a distance on probability measures is
$$
\|\mu-\nu\|_{s,p} = W_{0,\infty}(\mu,\nu) + W_{s,p}(\mu,\nu)
$$
with $s>0$ and $p<\infty$. Remark that this distance is defined only if $\|x\|_{H^s}^p$ is $\mu$ and $\nu$ integrable and if $\|x\|_{L^2}$ is in $L^\infty_\mu$ and $L^\infty_\nu$. We call $M_{s,p}$ the set of measures satisfying these properties. We choose these metrics because they correspond to weak convergence of the measures (like total variation distance) and convergence of the moments of orders $q$ and $r$
$$
\|x\|_{L^q_\mu,H^s}^q \; , \; \|x\|_{L^r_\mu,L^2}^r
$$
for all $q\leq p$ and all $r$, which gives some more information on the law of $\mu^t$. Where this continuity comes from may be more understandable if we consider random variables instead of measures. In terms of random variables, there is an analogy between the metric space $M_{s,p}$ and, given a probabilistic space $(\Omega, \mathcal A,\mathbb P)$, the space $L^\infty(\Omega,L^2(\T))\cap L^p(\Omega, H^s(\T))$. With two random variables $x$ and $y$ on $\Omega$ with values in $H^s$ almost surely, the continuity in $L^\infty(\Omega,L^2(\T))\cap L^p(\Omega, H^s(\T))$ requires on the one hand to bound $\Psi(t)x - \Psi(t)y$ in $L^\infty(\Omega,L^2(\T))$ in function of the norm of $x - y$. This bound comes from the conservation of the $L^2$-norm by the flow of KdV. On the other hand, we need to bound $\Psi(t)x - \Psi(t)y$ in $L^p(\Omega, H^s)$. The $L^p$ norm in probability is an obstacle. To get rid of this difficulty, we bound $\Psi(t)x_\omega - \Psi(t)y_\omega$ with $\omega$ the probability variable in $H^s$ by 
$$
C_1(x_\omega,y_\omega,t)\|x_\omega-y_\omega\|_{H^s} + C_2(x_\omega,y_\omega,t)(\|x_\omega\|_{H^s}+\|y_\omega\|_{H^s})\|x_\omega-y_\omega\|_{L^2} 
$$
with $C_1$ and $C_2$ depending on time and the $L^2$ norms of $x_\omega$ and $y_\omega$ such that we can take its $L^p$ norm in probability without losing any integrability. The strategy with Wasserstein metrics is the same.

Another reason why we choose Wasserstein metrics is that there exists an invariant measure under the flow of KdV in the intersection over $s<1/2$ and $p<\infty$ of the $M_{s,p}$. The stability of invariant measures in $M_{s,p}$, as we will define it, is a direct consequence of the continuity of the action of the flow on measures with regard to $\|\; .\;\|_{s,p}$. It means that at finite times, the distance between $\nu^t$ and its initial value is controlled by the distance between $\nu$ and the invariant measure.

Invariant measures are studied for many reasons. One of them is that they sometimes produce supercritical global well-posedness (see for instance \cite{BTranII}). Another one is that the invariance can be seen as an equilibrium of the system, weaker than thermodynamical equilibrium but stronger than statistical equilibrium (invariance of the mean values of the amplitudes of the Fourier coefficients of the solution, see \cite{zfwea}). They are usually built in the following way. We consider an invariant for a Hamiltonian equation (mass, energy ...) of the form 
$$
\|u\|_{H^s}^2 + R(u)
$$
where $u$ is the solution of the equation and $R(u)$ involves derivatives of $u$ weaker than $s$. The invariant measure resembles
$$
e^{-R(u)} e^{-\|u\|_{H^s}^2}du
$$
where
$$
d\mu(u) = e^{-\|u\|_{H^s}^2}du
$$
is a Gaussian measure with support in $H^{s-1/2-}$ in dimension 1. 

We exploit here three invariants of KdV. Actually, we consider the measure built and proved to be invariant by Bourgain in the appendix of \cite{Binvkdv}. We present the proof of the invariance in a wish of completeness rather than to claim any novelty about it. However, we focus on proving the invariance on the whole topological $\sigma$-algebra of $H^{1/2-} = \cap H^s$, $s<1/2$ instead of $H^{1/2-\varepsilon}$. 

The first invariant is the mean value along $x$
$$
\int_{\T} u(x,t) dx
$$
in order to assume that the solution has $0$ mean value along $x$ and build the Gaussian measure as the one induced by the random variable
$$
\sum_{n\neq 0} \frac{g_n}{|n|}e^{inx}
$$
where $g_n$ are Gaussian variables. This measure corresponds to
$$
e^{-\|\partial_ x u\|_{L^2}^2/2}du\; .
$$
Under the assumption of $0$ mean value, the $L^2$ norm of $\partial_x u$ corresponds to the $H^1$ norm of $u$, and the support of this measure is included in $H^{1/2-}$. 

The second one is the Hamiltonian
$$
\frac{1}{2}\|\partial _x \|_{L^2}^2 - \frac{1}{6}\int_{\T} u^3\; .
$$
Due to the absence of sign of $\int u^3$, it is unclear whether $e^{\int u^3}$ is $\mu$ integrable or not, which prevents us from using the measure $e^{\frac{1}{6} \int u^3} e^{-\frac{1}{2} \|\partial_x u \|_{L^2}^2}du$, given that we are looking for a measure in $M_{s,p}$. Hence, we use a third invariant, which is the invariance of the mass to write 
$$
d\rho(u) = 1_{\|u\|_{L^2}\leq 1} e^{-\frac{1}{2}\|\partial _x \|_{L^2}^2 + \frac{1}{6}\int_{\T} u^3}  du \; .
$$
This makes $\rho$ an invariant measure belonging to $M_{s,p}$ up to a renormalization factor. The invariance of $\rho$ comes from the preservation of the mentioned quantities and the invariance of Lebesgue measure under Hamiltonian flows. 

We prove the following result.

\begin{theorem}\label{th-mainresult} The action of the flow of KdV is continuous on $M_{s,p}$ according to the estimate, valid for all times $t\in \R$ and all measures $\mu,\nu \in M_{s,p}$,
$$
\|\mu^t - \nu^t \|_{s,p} \leq C (1+\|x\|_{L^p_\mu, H^s(\T)}+\|x\|_{L^p_\nu, H^s(\T)})e^{c|t|(\|x\|_{L^\infty_\mu, L^2(\T)}+\|x\|_{L^\infty_\nu, L^2(\T)})^{12}}\|\nu - \mu\|_{s,p} \; .
$$
\end{theorem}

Moreover, since the invariant measure $\rho$ belongs to $M_{s,p}$, we get the following corollary.

\begin{corollary}\label{cor-mainapp}The measure $\rho$ is locally stable in time, in the sense that for all times $t \in \R$ and for all measures $\nu \in M_{s,p}$, we have 
$$
\|\nu^t - \nu \|_{s,p} \leq Ce^{c|t|(1+\|x\|_{L^\infty_\nu,L^2})^{12}}(1+\|x\|_{L^p_\nu,H^s})\|\nu - \rho\|_{s,p} \; .
$$
\end{corollary}

The paper is organized as follows.

In Section 2, we present different metrics on probability measures and discuss their relevance to study the flows of non linear PDEs.

In Section 3, we prove local well-posedness for KdV in $H^s$ with $s\geq 0$ with a time of existence and uniqueness depending only on the $L^2$ norm of the initial datum. We deduce from it global well-posedness and useful estimates.

In Section 4, we prove the continuity of the action of the flow of KdV on the metric space of measures $M_{s,p}$.

In Section 5, we build and prove the invariance of the measure $\rho$ under the flow of KdV.

\paragraph{Acknowledgements} The first author is supported by the FIRB 2012 "Dispersive dynamics, Fourier analysis and variational methods".

The authors would like to thank Armen Shirikyan for helpful suggestions.

\section{Different metrics on probability measures}

In this section, we discuss different kinds of distances on probability measures which are commonly used, in transportation theory, for instance, see \cite{BKthMK} and references therein for a survey, and their relevance regarding the study of the continuity of the flow of Hamiltonian non linear equations. For further information on these distances, we also refer to \cite{KSmat,Drea}. We want to deduce the stability of invariant measures for these equations from this continuity. First, we introduce the total variation distance. The continuity of the flow for this distance is obtained with only few assumptions on the flow, but it only corresponds to weak convergence and we wanted to add the convergence of the moments of order $p$ too. Then, we introduce dual Lipschitz and Kantorovitch distances. We will explain in which sense these distances are not fit to study the continuity of non linear PDEs. Finally, we present the Wasserstein distances. These distances are defined only on certain measures, but the invariant measure we consider in Section \ref{sec-cons} is one of these. Besides, we prove in Section \ref{sec-contflow} that the flow of KdV is continuous with regard to these distances.

\subsection{Total variation, dual Lipschitz and Kantorovitch distances}

We assume here that $X$ is a Polish space (completely metrizable, separable, topological space) of functions with metrics $d_X$ and that $\Psi(t)$ is the flow of some equation such that $\Psi(t)$ is Lipschitz continuous on $X$ with some constant $C(t)$, that is
$$
Lip(\Psi(t)) = \sup_{x_1,x_2\in X} \frac{d_X(\Psi(t)x_1,\Psi(t)x_2)}{d_X(x_1,x_2)} = C(t)\; .
$$

For any $f : X\rightarrow \R$, we denote, if $f$ is bounded,
$$
\|f\|_\infty = \sup_{x\in X} |f(x)| \; ,
$$
if $f$ is Lipschitz continuous,
$$
Lip(f) = \sup_{x_1,x_2\in X} \frac{|f(x_1)-f(x_2)|}{d_X(x_1,x_2)}\; ,
$$
and if $f$ is both
$$
\|f\|_L = \|f\|_\infty+ Lip(f) \; .
$$

Let us define usual metrics on probability measures on the topological $\sigma$-algebra of $X$.

In the following, we let $\mu$, $\nu$ be two measures on $X$.

\begin{definition}[Total variation distance] We call the total variation distance and we write $\|\mu - \nu\|_{\textrm{var}}$, the distances
$$
\|\mu-\nu\|_{\textrm{var}} = \frac{1}{2} \sup_{\|f\|_\infty \leq 1} \Big|\langle f\; ,\; \mu\rangle - \langle f\; ,\; \nu\rangle \Big|\; ,
$$
where $\langle \; ,\; \rangle $ denotes $\int f d\mu$.
\end{definition}

We then call $\mu^t$ the measure on $X$ defined as the pushforward measure of $\mu$ through the flow $\Psi(t)$ that is  
$$
\mu^t(A) = \mu( \Psi(t)^{-1}A)
$$
for any measurable set $A$, which is defined as long as $\Psi(t)$ is a measurable function from $X$ to $X$.

Then, it appears that it suffices for $\Psi(t)$ to be measurable to get 
$$
\|\mu^t - \nu^t\|_{\textrm{var}} \leq \|\mu - \nu\|_{\textrm{var}}
$$
since, if $\|f\|_\infty \leq 1$ then $\|f\circ \Psi(t) \|_\infty \leq 1$, and by definition, we have that
$$
\langle f ,\mu^t\rangle = \langle f\circ \Psi(t) ,\mu\rangle \; .
$$

\begin{definition}[Dual Lipschitz distance]We call the dual Lipschitz distance and we write $\|\mu - \nu\|_L^*$, the distances
$$
\|\mu-\nu\|_L^* = \sup_{\|f\|_L \leq 1} \Big|\langle f\; ,\; \mu\rangle - \langle f\; ,\; \nu\rangle \Big|\; .
$$
\end{definition}

We then get that 
$$
\|\mu^t-\nu^t\|_L^* \leq (1+C(t)) \|\mu-\nu\|_L^* \; ,
$$
since, if $\|f\|_L \leq 1$, then $\|f\circ \Psi(t)\|_\infty \leq \|f\|_\infty \leq 1$ and $Lip(f\circ \Psi(t)) \leq Lip(f) Lip(\Psi(t)) \leq C(t)$, which implies that
$$
\|f\circ \Psi(t)\| \leq 1+C(t)\; .
$$
Hence, $\frac{f\circ \Psi(t)}{1+C(t)}$ is Lipschitz continuous with constant less than $1$, and we can conclude.

\begin{definition} Let $\mathcal M_1$ be the space of measures $\mu$ on $X$ such that for all $x_0$, $d_X(x,x_0)$ is $\mu$ integrable. We call the Kantorovitch distance and we write $\|\mu - \nu\|_K$, the distances on $\mathcal M_1$ 
$$
\|\mu-\nu\|_{K} = \sup_{Lip(f) \leq 1} \Big|\langle f\; ,\; \mu\rangle - \langle f\; ,\; \nu\rangle \Big|\; ,
$$
\end{definition}

From previous remarks, we see that
$$
\|\mu^t - \nu^t\|_K\leq C(t) \|\mu-\nu\| \; .
$$

Then, it also happens that if $\mu$ is an invariant measure through the flow $\Psi(t)$ then as 
$$
\langle f, \mu^t\rangle  = \langle f ,\mu\rangle 
$$
for all $f$ bounded or for all $f$ Lipschitz continuous if $\mu$ is in $\mathcal M_1$, we get that
$$
\|\mu^t-\mu\|_K = \|\mu^t-\mu\|_L^*=\|\mu^t-\mu\|_{\textrm{var}} = 0\; .
$$

The problem with these distances and their different previously mentioned properties with regard to the flow of a non linear PDE such as KdV is that this flow has no reason to be Lipschitz continuous with a constant depending only on time and not on the size of the initial datum. However, as it will be proved, one can build an invariant measure $\mu$ for the flow of KdV on $H^{1/2-}$ such that $\|x\|_{L^2}$ is in $L^\infty_\mu$ and $\|x\|_{H^s}$ is in $L^p_\mu$ for all $s<\frac{1}{2}$ and all $p<\infty$.

\subsection{Wasserstein metrics}

The Wasserstein distance is defined only on certain measures. The space onto which the version of the Wasserstein metrics we use is defined is given in the next definition.

\begin{definition}[Spaces $M_{s,p}$] For $s\geq 0$ and $p\in [1,\infty]$, let $M_{s,p}$ be the space of measures $\mu$ on the topological $\sigma$-algebra of $H^s(\mathbb T)$ where $\mathbb T$ is the torus of dimension 1 such that $\|u\|_{L^2}$ belongs to $L^\infty_\mu$ and $\|u\|_{H^s}$ belongs to $L^p_\mu$. \end{definition}

\begin{remark} Measures $\mu$ on $H^s$ that satisfy large Gaussian deviation estimates (for instance Gibbs measures, and several invariant measures for various flows of Hamiltonian PDEs, see \cite{TVinvBO,NORSinv,BTTlong,dSinv}) are such that $\|u\|_{H^s}$ belong to $L^p_\mu$ for any $p<\infty$. \end{remark}

Let us now introduce the Wasserstein metrics.

\begin{definition}[Wasserstein metrics] For all $\mu,\nu \in M_{s,p}$, we call the $s,p$-Wasserstein distance and we denote $W_{s,p}$, the distance
$$
W_{s,p}(\mu,\nu) = \left( \inf \lbrace \int \|x-y\|_{H^s}^p d\gamma(x,y)\; \Big|\; \gamma \in \textrm{Marg}(\mu,\nu)\rbrace\right)^{1/p} \; ,
$$
where $\textrm{Marg}(\mu,\nu)$ is the set of measures on the topological $\sigma$-algebra on $H^s\times H^s$ whose marginals are $\mu$ and $\nu$.

For $p=\infty$, we define
$$
W_{s,\infty} = \inf \lbrace \|\;\|x-y\|_{H^s}\|_{L^\infty_\gamma} \; \Big|\; \gamma \in \textrm{Marg}(\mu,\nu) \rbrace \; .
$$
\end{definition}

\begin{remark} The distance $W_{s,p}$ corresponds to weak convergence of the measure plus the convergence of the moments of order $q\leq p$ defined as 
$$
\int \|x\|_{H^s}^pd\mu(x) \; .
$$

Besides, the Rubinstein-Kantorovitch theorem provides an equivalence for the distance $W_{s,1}$. \end{remark}

\begin{theorem}[Rubinstein-Kantorovitch]The distance $W$ defined as
$$
W(\mu,\nu) = \inf_{\gamma \in \textrm{Marg}(\mu,\nu)} \int d_X(x,y) d\gamma(x,y)
$$
is equivalent to the Kantorovitch distance. Besides, the infimum is reached if $\mu$ and $\nu$ are tight.\end{theorem}

For the proof and comments, see \cite{Drea}.

From now on, the distance we adopt to compare measures is defined as 
$$
\|\mu-\nu\|_{s,p} =  W_{0,\infty} (\mu,\nu) + W_{s,p}(\mu,\nu) \; .
$$

To make a last remark in this section and in view of what we said in the introduction about random variables, to prove the continuity of the flow on $M_{s,p}$, for all $s\geq 0$, it suffices to prove an inequality of the kind 
$$
\|\Psi(t)u-\Psi(t)v\|_{H^s} \leq C_1(t,u,v) \|u-v\|_{H^s} + C_2(t,u,v) (\|u\|_{H^s}+\|v\|_{H^s})\|u-v\|_{L^2}\; ,
$$
where $C_i(t,u,v)$ depends on time and the $L^2$ norms of $u$ and $v$.

 \section{Deterministic properties of KdV}
 
In this section, we present a sketch of the general theory of well-posedness for the periodic KdV equation providing some bilinear estimates that will be important in the forthcoming. We do not pretend to be exhaustive here, so some technical lemmas will be admitted and some proofs just sketched, referring to \cite{bourgain}, \cite{KPVbil}, \cite{CKSTTsha} and references therein for details and further results.

The Cauchy problem for the KdV equation on the torus reads as
\begin{equation}\label{pb}
\begin{cases}
\partial_tu+\partial_x^3u+\frac12\partial_xu^2=0, \quad x\in\T\\
u(x,0)=u_0(x).
\end{cases}
\end{equation}

The corresponding integral equation is thus the following one
\begin{equation}\label{intform}
u(t)=S(t)u_0-\int_0^tS(t-t')\frac12\partial_x(u^2)(t')dt'
\end{equation}
where $S(t)=e^{-t\partial_x^3}$ is the flow of the linearized around $0$ equation. We will formally denote by $u=\Psi(t)u_0$ the solution to \eqref{pb} with initial condition $u_0$. Space-time Fourier transform allows us to explicitly write the linear propagator $S(t)$ as
\begin{equation}\label{s}
S(t)u_0\sim\sum_{k\in\mathbb{Z}}\int e^{ i(kx + \tau t)}\delta(\tau-k^3)\hat{u_0}(k)d\tau,
\end{equation}
$\delta(x)$ being the 1-dimensional Dirac mass. 

Notice that this formula shows that the linear solution of \eqref{pb} has its space-time Fourier transform supported on the cubic $\tau=k^3$. In \cite{bourgain} the author showed that, after a localization in time, the Fourier transform of the nonlinear solution still concentrates near the same cubic. This feature suggests the introduction of the following functional spaces (called \emph{Bourgain spaces}).

\begin{definition}
We denote by $X^{s,b}$, $Y^s$, $Z^s$ the spaces of functions $u:\T\times\mathbb{R}\rightarrow\mathbb{R}$ with mean value zero (i.e. $\int_{\mathbb{T}}u(x,t)dx=0$) such that the corresponding norms 
\begin{equation}\label{x}
\|u\|_{X^{s,b}}=
\||k|^s\langle\tau-k^3\rangle^b\hat{u}(k,\tau)\|_{l^2_k,L^2_\tau},
\end{equation}
\begin{equation}\label{y}
\|u\|_{Y^s}=\|u\|_{X^{s,\frac12}}+
\||k|^s\hat{u}(k,\tau)\|_{l^2_k,L^1_\tau},
\end{equation}
\begin{equation}\label{z}
\|u\|_{Z^s}=\|u\|_{X^{s,-\frac12}}+
\left\|\frac{|k|^s\hat{u}(k,\tau)}{\langle\tau-k^3\rangle}\right\|_{l^2_k,L^1_\tau}
\end{equation}
are finite (we are denoting with $\displaystyle\|f\|_{l^2_k}=\sum_{k\in\mathbb{Z}}|\hat{f}(k)|^2)$.
\end{definition}

\begin{remark}
The conservation of the spatial mean allows us to assume that the initial data $u_0$ satisfies a mean-zero assumption. This fact will be important in the forthcoming, and moreover it makes the homogeneous and non homogeneous versions of the Bourgain spaces (i.e. with the weights $|k|^s$ or $\langle k\rangle^s$) equivalent.
\end{remark}

\begin{remark}
Here and in the following we shall just deal with the case $s\geq0$, referring to \cite{KPVbil} and \cite{CKSTTsha} for more general results in the negative case.
\end{remark}

\subsection{Local theory}

We present here a local existence result for the periodic KdV equation that will rely on a contraction argument on the Bourgain spaces. In order to make our estimates work, as already pointed out, we shall need to apply a smooth cut-off function in order to localize the solution in time.
In the sequel $\eta(t)$ will thus represent a smooth bump function supported in $[-2,2]$, with $\eta=1$ in $[-1,1]$, and we will denote with $\eta_T=\eta(t/T)$ the corresponding rescaled function.

In the following two propositions we collect the crucial estimates needed to make the contraction argument work.

\begin{proposition}[Linear estimates]\label{linp}
For every $s\geq 0$, there exist $C_1$ and $C_2$ such that for every $\Phi$ and every $F$ supported in $\T\times [-3,3]$, the following estimates hold
\begin{equation}\label{lin1}
\|\eta(t)S(t)\phi\|_{Y^s}\leq C_1\|\phi\|_{H^s}
\end{equation}
\begin{equation}\label{lin2}
\left\|\eta(t)\int^t_0S(t-t')F(t')dt'\right\|_{Y^s}\leq C_2 \|F\|_{Z^s}
\end{equation}
with constants $C_1$, $C_2$ independent of $\phi$ and $F$.
\end{proposition}

\begin{proposition}[Bilinear estimates]\label{bilp}
For every $s\geq 0$, there exist $C_3$ and $C_4$ such that the following estimates hold
\begin{equation}\label{bil}
\|\eta(t)\partial_x(uv)\|_{Z^s}\leq C_3\left(\|u\|_{Y^s}\|v\|_{Y^0}+\|u\|_{Y^0}\|v\|_{Y^s}\right),
\end{equation}
\begin{equation}\label{bil2}
\|\eta_T(t)\partial_x(uv)\|_{Z^s}\leq C_4 T^{1/12}\left(\|u\|_{Y^s}\|v\|_{Y^0}+\|u\|_{Y^0}\|v\|_{Y^s}\right).
\end{equation}
with constants $C_3$, $C_4$ independent of $T$, $u$ and $v$.
\end{proposition}

\begin{proof}\emph{Linear estimates.}
We begin by writing, by definition,
$$
\eta(t)S(t)\phi = \eta(t)\sum_k\hat{\phi}(k)e^{i(kx+k^3t)}
$$
and by writing $\eta$ in terms of its Fourier transform and doing a change of variable we get
$$
\eta(t)S(t)\phi = c \sum_k\hat{\phi}(k)e^{ikx}\int_{-\infty}^{+\infty}e^{it\tau}\hat{\eta}(\tau-k^3)d\tau.
$$
Thus we have
\begin{eqnarray*}
\|\eta S(t)\phi\|_{X^{s,\frac12}}^2&\leq&
c\sum_k|k|^{2s}|\hat{\phi}(k)|^2\int_{-\infty}^{+\infty}(1+|\tau-k^3|)|\hat{\eta}(\tau-k^3)|^2d\tau
\\&\leq&
c\|\phi\|_{H^s}^2,
\end{eqnarray*}
and
\begin{eqnarray*}
\|\, |k|^s\mathcal{F}(\eta(t)S(t)\phi)\|_{L^2(dk)L^1(d\tau)}^2&\leq&
c\sum_k|k|^{2s}|\hat{\phi}(k)|^2\left(\int_{-\infty}^{+\infty}|\hat{\eta}(\tau-k^3)|d\tau\right)^2
\\&\leq&
c\|\phi\|_{H^s}^2,
\end{eqnarray*}
where $\mathcal F$ is the Fourier transform and this concludes the proof of \eqref{lin1}.

Turning to \eqref{lin2}, we write 

\begin{equation}\label{lin22}
\eta(t)\int^t_0S(t-t')F(t')dt' = I_1 + I_2
\end{equation}
with 
$$
I_1 = \eta(t)S(t)\int_{\mathbb{R}}a(t')S(-t')F(t')dt'
$$
and
$$
I_2 = \eta(t)\int_{\mathbb{R}}a(t-t')S(t-t')F(t')dt'
$$
where we have set $a(t)=\rm{sgn}(t)\tilde{\eta}(t)$ with $\tilde{\eta}$ being a smooth cut-off supported in $[-10,10]$ and equal to $1$ in $[-5,5]$. With this choice we have indeed, for all $t\in[-2,2]$ and $t'\in[-3,3]$,
\begin{equation*}
\chi_{[0,t]}(t')=\frac12(a(t')-a(t-t'))
\end{equation*}
so that \eqref{lin22} holds (we assumed that $F$ was supported in $\mathbb{T}\times[-3,3]$). 

We now estimate the righthandside of \eqref{lin22} term by term. Due to \eqref{lin1}, to estimate the contribution given by $I_1$, it is enough to show that
\begin{equation}\label{i1}
\left\|\int_{\mathbb{R}}a(t')S(-t')F(t')dt'\right\|_{H^s}\lesssim\|F\|_{Z^s}.
\end{equation}
Using Fourier transform and recalling \eqref{z} we have
\begin{eqnarray*}
\left\|\int_{\mathbb{R}}a(t')S(-t')F(t')dt'\right\|_{H^s}&=&
\left\||k|^s\mathcal{F}\left(\int_{\mathbb{R}}a(t')S(-t')F(t')dt'\right)(k)\right\|_{l^2_k}
\\
&=&
\left\||k|^s\int_{\mathbb{R}}\hat{a}(\tau-k^3)\hat{F}(k,\tau)d\tau\right\|_{l^2_k}\; .
\end{eqnarray*}

Since $|\hat{a}(\tau)|=O(\langle\tau\rangle^{-1})$, the proof of \eqref{i1} is concluded. Let us now turn to $I_2$. Neglecting the cutoff $\eta(t)$ and space-time Fourier transforming yields
\begin{equation*}
\mathcal{F}\left(\int a(t-t')S(t-t')F(t')dt'\right)(k,\tau)\sim\hat{a}(\tau-k^3)\hat{F}(k,\tau).
\end{equation*}
The claimed estimate then follows from \eqref{y}, \eqref{z} and the decay estimate for $\hat{a}$ used above.

\end{proof}

\emph{Proof of Bilinear estimates.}
We adapt the proof of Lemmas (7.41) and (7.42) in \cite{bourgain} modifying it in order to obtain the estimates we need. In this part we shall neglect the cutoff in order to simplify the presentation, the local case being obtained by standard regularization arguments.
Writing $w=\partial_x(uv)$ we thus need to estimate
\begin{equation}\label{norms}
\|w\|_{Z^s}=\left(\sum_{k\neq0}|k|^{2s}\int_{-\infty}^{+\infty}\frac{|\hat{w}(k,\tau)|^2}{1+|\tau-k^3|}d\tau\right)^\frac12+
\left(\sum_{k\neq0}|k|^{2s}\left(\int_{-\infty}^{+\infty}\frac{|\hat{w}(k,\tau)|}{1+|\tau-k^3|}d\tau\right)^2\right)^\frac12.
\end{equation}

First of all, notice that
\begin{equation*}
\hat{w}(k,\tau)=i k\:\widehat{uv}(k,\tau)=ik(\hat{u}\ast\hat{v})(k,\tau)
\end{equation*}
where $*$ denotes the standard convolution product, 
so that
\begin{equation}\label{west}
|\hat{w}(k,\tau)|\leq |k|\sum_{k_1\neq0}\int d\tau_1\big(|\hat{u}(k_1,\tau_1)|\:|\hat{u}(k-k_1,\tau-\tau_1)|\big).
\end{equation}
We define for every $s\geq 0$
\begin{equation}\label{c1}
c_s(k,\tau)=|k|^s(1+|\tau-k^3|)^{1/2}|\hat{u}(k,\tau)|,
\end{equation}
\begin{equation}\label{c2}
d_s(k,\tau)=|k|^s(1+|\tau-k^3|)^{1/2}|\hat{v}(k,\tau)|
\end{equation}

We now recall the following result from \cite{bourgain} that will be of crucial importance in the sequel.

\begin{proposition}\label{bour}
Let $f$ be a function defined on the torus $\mathbb{T}^2$. The following estimates hold
\begin{equation}\label{cit}
\|f\|_{L^4(\mathbb{T}^2)}\lesssim\left(\sum_{m,n\in\mathbb{Z}}(1+|n-m^3|)^{2/3}|\hat{f}(m,n)|^2\right)^{1/2},
\end{equation}
\begin{equation}\label{cit2}
\left(\sum_{m,n\in\mathbb{Z}}(1+|n-m^3|)^{-2/3}|\hat{f}(m,n)|^2\right)^{1/2}\lesssim\|f\|_{L^{4/3}(\mathbb{T}^2)}.
\end{equation}
\end{proposition}

\emph{Proof of Lemma \ref{bour}}
See \cite{bourgain}.
\begin{remark}
We shall in fact use this result in the following versions, which are implied by \eqref{cit} and \eqref{cit2} since $\eta$ has compact support included in $]-\pi,\pi[$. We do not make any assumption on the temporal support of $f$.
\begin{equation*}
\|\eta(t)f\|_{L^4(\mathbb{R}\times\mathbb{T})}\lesssim\left(\int d\tau\sum_{k\in\mathbb{Z}}(1+|\tau-k^3|)^{2/3}|\hat{\eta f}(\tau,k)|^2\right)^{1/2},
\end{equation*}
and
\begin{equation*}
\left(\int_{\R}\sum_{k\in \Z} (1+|\tau-k^3|)^{-2/3}|\hat{\eta f}(\tau,k)|^2d\tau\right)^{1/2}\lesssim\|\eta(t)f\|_{L^{4/3}(\mathbb{T}^2)}.
\end{equation*}

\end{remark}
The rest of the proof is essentially contained in the ones of the following two lemmas.

\begin{lemma}\label{1} For every $u$ and $v$ in $Y^s$ and with $w = \partial_x (uv)$ we have
\begin{equation}\label{ll1}
\left(\sum_{k\neq0}|k|^{2s}\int_{-\infty}^{+\infty}\frac{|\hat{w}(k,\tau)|^2}{1+|\tau-k^3|}d\tau\right)^\frac12\lesssim
\|u\|_{Y^s}\|v\|_{Y^0}+\|u\|_{Y^0}\|v\|_{Y^s} \; .
\end{equation}
\end{lemma}

\begin{lemma}\label{2} For every $u$ and $v$ in $Y^s$ and with $w = \partial_x (uv)$ we have
\begin{equation}\label{ll2}
\left(\sum_{k\neq0}|k|^{2s}\left(\int_{-\infty}^{+\infty}\frac{|\hat{w}(k,\tau)|}{1+|\tau-k^3|}d\tau\right)^2\right)^\frac12\lesssim\|u\|_{Y^s}\|v\|_{Y^0}+\|u\|_{Y^0}\|v\|_{Y^s}\; .
\end{equation}
\end{lemma}

Notice that these two results, together with \eqref{norms}, yield \eqref{bil} and thus conclude the proof.

\begin{proof}\emph{Lemma \ref{1}.}
From \eqref{west}, \eqref{c1} and \eqref{c2}  we have
\begin{equation*}
\frac{|k|^s||\hat{w}(k,\tau)|}{(1+|\tau-k^3|)^{1/2}}
\end{equation*}
\begin{equation}\label{ccc}
\leq\sum_{k_1}\int d\tau_1\frac{|k|^{s+1}|k_1|^{-s}|k-k_1|^{-s}c_s(k_1,\tau_1)d_s(k-k_1,\tau-\tau_1)}
{(1+|\tau-k^3|)^{1/2}(1+|\tau_1-k_1^3|)^{1/2}(1+|\tau-\tau_1-(k-k_1)^3)|^{1/2}}.
\end{equation}
Notice that since 
\begin{equation}\label{trick}
|k|^s\leq C_s(|a|^s+|b|^s)\quad \forall a+b=k,
\end{equation} 
picking $a=k_1$ and $b=k-k_1$ we can estimate \eqref{ccc} with
\begin{equation}
C_s\sum_{k_1}\int d\tau_1|k|\frac{(|k_1|^{-s}+|k-k_1|^{-s})c_s(k_1,\tau_1)d_s(k-k_1,\tau-\tau_1)}
{(1+|\tau-k^3|)^{1/2}(1+|\tau_1-k_1^3|)^{1/2}(1+|\tau-\tau_1-(k-k_1)^3)|^{1/2}}
\end{equation}
\begin{equation}\label{uu}
= \sum_{k_1}\int d\tau_1|k|\frac{c_0(k_1,\tau_1)d_s(k-k_1,\tau-\tau_1)+c_s(k_1,\tau_1)d_0(k-k_1,\tau-\tau_1)}
{(1+|\tau-k^3|)^{1/2}(1+|\tau_1-k_1^3|)^{1/2}(1+|\tau-\tau_1-(k-k_1)^3)|^{1/2}}.
\end{equation}
By assumption on $u$, we have that $c_s(0,\tau)=d_s(0,\tau)=0$, so that we may assume $k$, $k_1$ and $k-k_1\neq0$ in \eqref{ccc}. In this situation we have
\begin{equation}\label{k2}
|(\tau-k^3)-[(\tau_1-k_1^3)+(\tau-\tau_1-(k-k_1)^3)]|=|3k_1(k-k_1)k|\geq\frac32k^2.
\end{equation}

We now estimate the sum \eqref{uu} dividing the indexes into three sets as follows:
\begin{equation*}
A=\{k,k_1\neq0:|\tau-k^3|\geq\frac32k^2\}
\end{equation*}
\begin{equation*}
B=\{k,k_1\neq0:|\tau_1-k_1^3|\geq\frac32k^2\}
\end{equation*}
\begin{equation*}
C=\{k,k_1\neq0:|\tau-\tau_1-(k-k_1)^3|\geq\frac32k^2\}.
\end{equation*}

Due to \eqref{k2} we can thus estimate, since we are summing positive terms,
\begin{equation*}
\sum_{k,k_1\neq0}\leq\sum_{k,k_1\in A}+\sum_{k,k_1\in B}+\sum_{k,k_1\in C}.
\end{equation*}
and analyze the three sums one by one. We limit ourselves to consider the terms with $c_0d_s$, the other one been analogous by symmetry.

\textbf{Set A}. 
In this case we have that since
\begin{equation}\label{prima}
\sum_{k_1\neq0}\int d\tau_1\frac{c_0(k_1,\tau_1)d_s(k-k_1,\tau-\tau_1)}
{(1+|\tau_1-k_1^3|)^{1/2}(1+|\tau-\tau_1-(k-k_1)^3)|^{1/2}}=\widehat{F\cdot G}(k,\tau)
\end{equation}
where we have set
\begin{equation}\label{F}
F(x,t)=\sum_m\int d\mu\left\{e^{i(mx+\mu t)}\frac{c_0(m,\mu)}{(1+|\mu-m^3|)^{1/2}}\right\}
\end{equation}
and
\begin{equation}\label{G}
G(x,t)=\sum_m\int d\mu\left\{e^{i(mx+\mu t)}\frac{d_s(m,\mu)}{(1+|\mu-m^3|)^{1/2}}\right\}
\end{equation}
(notice that $\|F\|_{L^2}=\|u\|_{X^{0,\frac12}}$ and  $\|G\|_{L^2}=\|v\|_{X^{s,\frac12}}$),
this first contribution to the left member of \eqref{ll1} is at most
\begin{equation}
\left(\sum_{k\in A}\int|\widehat{F\cdot G}(k,\tau)|d\tau\right)^{1/2}\leq\|F\cdot G\|_{L^2_{x,t}}
\lesssim\|F\|_{L^4_{x,t}}\|G\|_{L^4_{x,t}}. 
\end{equation}
Since $2/3<1$ estimate \eqref{cit} implies that
\begin{eqnarray}\label{ccc1}
\|F\|_{L^4}\|G\|_{L^4}&\lesssim&\left(\sum_m\int c_0(m,\mu)^2d\mu\right)^{1/2}
\left(\sum_m\int d_s(m,\mu)^2d\mu\right)^{1/2}\\
&\lesssim&\nonumber
\|u\|_{Y^0}\|v\|_{Y^s} \; .
\end{eqnarray}

\textbf{Set B}.
Analogously to the previous case, the contribution to the left member of \eqref{ll1} is thus given by
\begin{equation}\label{yeah}
\left\{\sum_{k\in B}\int d\tau\left(\frac1{(1+|\tau-k^3|)^{1/2}}\widehat{F\cdot H}(k,\tau)\right)^2\right\}^{1/2}
\end{equation}
where $F$ is given by \eqref{F} and 
\begin{equation}\label{H}
H(x,t)=\sum_m\int d\mu\big[e^{i(mx+\mu t)}d_s(m,\mu)\big].
\end{equation}
Same considerations as in the previous case and the application of Proposition \ref{bour} allow us to estimate \eqref{yeah} with $\|F\cdot H\|_{L^{4/3}_{x,t}}$. H\"older inequality and Proposition \ref{bour} eventually yields
\begin{equation}\label{ccc2}
\|F\cdot H\|_{L^{4/3}_{x,t}}\lesssim\|F\|_{L^4}\|H\|_{L^2}\lesssim\|u\|_{Y^0}\|v\|_{Y^s}.
\end{equation}

\textbf{Set C}.
Similar to set B.

This concludes the proof of Lemma \ref{1}.\end{proof}

\begin{proof}\textit{Lemma \ref{2}. }
We consider now the term
\begin{equation*}
\frac{|k|^s||\hat{w}(k,\tau)|}{1+|\tau-k^3|}
\end{equation*} 
which we can bound by
\begin{equation}\label{cccd}
\sum_{k_1}\int d\tau_1\frac{|k|^{s+1}|k_1|^{-s}|k-k_1|^{-s}c_s(k_1,\tau_1)d_s(k-k_1,\tau-\tau_1)}
{(1+|\tau-k^3|)(1+|\tau_1-k_1^3|)^{1/2}(1+|\tau-\tau_1-(k-k_1)^3)|^{1/2}}.
\end{equation}
that can be estimated using \eqref{trick} with
\begin{equation}\label{uud}
C_s \sum_{k_1}\int d\tau_1|k|\frac{c_0(k_1,\tau_1)d_s(k-k_1,\tau-\tau_1)+c_s(k_1,\tau_1)d_0(k-k_1,\tau-\tau_1)}
{(1+|\tau-k^3|)(1+|\tau_1-k_1^3|)^{1/2}(1+|\tau-\tau_1-(k-k_1)^3)|^{1/2}}.
\end{equation}
We separate again the indexes into the three sets $A$, $B$, and $C$ as before, and limit ourselves by symmetry  to the terms with $c_0d_s$.

\textbf{Set A}.
In this case, we use a duality argument. In the set A, we can bound $\frac1{1+|\tau-k^3|}$ by $\frac{5}{3(k^2+|\tau-k^3|)}$. Consider a sequence $\{a_n\}$ such that
$$
a_n\geq0,\quad \sum_na_n^2=1,
$$
then by duality and \eqref{prima}-\eqref{G} we can estimates the left member of \eqref{ll2} by considering the scalar product in $L^2_\tau, l^2_k$
\begin{equation}\label{prod}
\sum_n\int d\tau \left(\frac{a_n|n|}{n^2+|\tau-n^3|}\widehat{F\cdot G}\right).
\end{equation}
Setting now 
\begin{equation*}
P(x,t)=\sum_n\int d\tau e^{i(nx+\tau t)}\frac{a_n|n|}{n^2+|\tau-n^3|}
\end{equation*}
whose $L^2$ norm is given by
\begin{equation*}
\|P\|_{L^2}\cong\left(\sum_n\int d\tau\frac{a_n^2|n|^2}{(n^2+|\tau|)^2}\right)^{1/2}
\end{equation*}
allows us to rewrite \eqref{prod}
\begin{equation}\label{ccc3}
\langle P,F\cdot G\rangle \leq\|P\|_{L^2}\|F\cdot G\|_{L^2}\leq\|u\|_{Y^0}\|v\|_{Y^s}.
\end{equation}

\textbf{Set B}.
We thus consider the quantity
\begin{equation*}
 \sum_{k_1\in B}\int d\tau_1\frac{c_0(k_1,\tau_1)d_s(k-k_1,\tau-\tau_1)}
{(1+|\tau-k^3|)(1+|\tau-\tau_1-(k-k_1)^3)|^{1/2}}.
\end{equation*}
Fix now $1/3<\rho<1/2$ and write $(1+|\tau-k^3|)^{-1}=(1+|\tau-k^3|)^{-(1-\rho)}(1+|\tau-k^3|)^{-\rho}$. We thus can rewrite the left hand side of \eqref{ll2} as
\begin{equation}\label{bub}
\left\{\sum_{k\neq0}\left(\int d\tau\frac1{(1+|\tau-k^3|)^{1-\rho}}\sum_{k_1}\int d\tau_1\frac{c_0(k_1,\tau_1)d_s(k-k_1,\tau-\tau_1)}
{(1+|\tau-k^3|)^\rho(1+|\tau-\tau_1-(k-k_1)^3)|^{1/2}}\right)^2\right\}^{1/2}.
\end{equation}
We use Cauchy-Schwartz inequality for the $\tau$-integral to estimate \eqref{bub} by 
\begin{equation}\label{so}
\left\{ \sum_{k\neq0}\int d\tau\left(\sum_{k_1}\int d\tau_1\frac{c_0(k_1,\tau_1)d_s(k-k_1,\tau-\tau_1)}{(1+|\tau-k^3|)^\rho(1+|\tau-\tau_1-(k-k_1)^3|)^{1/2}}\right)^2\right\}^{1/2}
\end{equation}
where we have used the fact that $\int(1+|\tau-k^3|^{-(2-2\rho)}) d\tau$ is finite since $\rho<1/2$.

We now rewrite \eqref{so} as
\begin{equation*}
\left\{\sum_{k\in B}\int d\tau\left(\frac1{(1+|\tau-k^3|)^{\rho}}\widehat{F\cdot H}(k,\tau)\right)^2\right\}^{1/2}
\end{equation*}
with $F$ and $H$ given by \eqref{F} and \eqref{H}. To conclude we proceed as for the set $B$ in the proof of Lemma \ref{1}; the only difference we have to care of is the power $\rho$ instead of $1/2$ for $1+|\tau-k^3|$ but every $\rho>1/3$ works.

\textbf{Set C}.
Similar to set $B$.

This concludes the proof of \eqref{bil}.

We finally turn to the proof of \eqref{bil2}.
We start by noticing that, given the definition of $F$ and $G$, that we can find in \eqref{F}-\eqref{G}, and thanks to estimate \eqref{cit}, we have

\begin{equation}\label{yu}
\|F\|_{L^4}\leq\left(\sum_m\int d\mu\frac{c_s(m,\mu)^2}{(1+|\mu-m^3|)^{1/3}}\right)^{1/2}
\end{equation}
and an analogous one for $\|G\|_{L^4}$ (notice that now we are replacing $u$ in definition \eqref{c1} with $\eta_Tu$). 
Rewriting the righthandside of \eqref{yu} as
\begin{equation*}
\left[\sum_m\int d\mu\left((1+|\mu-m^3|)^{1/2}|m|^s|\widehat{\eta_Tu}(m,\mu)|)\right)^{4/3}
\left(|m|^s|\widehat{\eta_Tu}(m,\mu)|)\right)^{2/3}\right]^{1/2}
\end{equation*}
and applying H\"older inequality with exponents $3/2$ and $3$ yields
\begin{equation}\label{oss}
\|F\|_{L^4}\leq\left(\sum_m\int d\mu \:c_s(m,\mu)^2\right)^{1/3}\||\partial_x|^s(\eta_Tu)\|_{L^2}^{1/3}
\end{equation}
where $|\partial_x|^s$ denotes the Fourier multiplier $\hat u(k)\mapsto |k|^s \hat u(k)$.
Since by H\"older inequality we can estimate $\||\partial_x|^s(\eta_Tu)\|_{L^2}\leq\|\eta_T\|_{L^4}\||\partial_x|^su\|_{L^4}$ and we can assume without loss of generality that $\|\eta_T\|_{L^p}\sim T^{1/p}$, we can continue estimating \eqref{oss} by
\begin{equation*}
\|u\|^{2/3}_{X^{s,\frac12}}T^{1/12}\|\partial_x^su\|_{L^4}.
\end{equation*}
We can now apply Proposition \ref{bour} to conclude that
\begin{eqnarray}\label{ecco}
\|F\|_{L^4}&\leq& T^{1/12}\|u\|^{2/3}_{X^{s,\frac12}}\left[\sum_n\int (1+|\tau-k^3|)^{2/3}|k|^{2s}|\hat{u}(k,\tau)|^2d\tau|\right]^{1/6}
\\\nonumber
&\leq&T^{1/12}\|u\|_{X^{s,\frac12}}
\end{eqnarray} 

Notice now that in the proof of \eqref{bil}, we bounded the non linearity $w$ by a product of either $\|F\|_{L^4}$, $\|G\|_{L^4}$, hence we can get an estimate on $w$ smaller than $T^{1/12}$

\end{proof}

We have now all the tools we need in order to prove LWP for problem \eqref{pb}.

\begin{proposition}\label{lwp}
For any $s\geq0$ the initial value problem \eqref{pb} is locally well-posed in \\
$\mathcal C([-T,T],H^s)\cap Y^s$, meaning that for any initial data $u_0\in H^s$ there exists $T=T(\|u_0\|_{L^2})>0$ such that there exists a unique solution
\begin{equation*}
u\in C([-T,T],H^s(\mathbb{T}))\cap Y^s
\end{equation*}
with
\begin{equation*}
\|1_{t \in [-T,T]} u\|_{Y^s}\leq C\|u_0\|_{H^s}.
\end{equation*}
\end{proposition}

\begin{remark} The $L^\infty([-T,T],H^s(\T))$ norm is controlled by the $Y^s$ norm.\end{remark}

\begin{proof}
Fix $u_0\in H^s$, $s\geq0$, and consider the map
\begin{equation}\label{map}
\Gamma(u)(t,x)=\eta_T(t)S(t)u_0-\eta_T(t)\int_0^tS(t-t')\left(\eta_T(t')\frac12\partial_x(u^2)(t')\right)dt'.
\end{equation}
(notice that choosing $T<1$ yields $\eta_T(t)\leq\eta(t)$ and so estimates \eqref{lin1} and \eqref{lin2} apply).

First of all, notice that if $u$ is a fixed point of $\Gamma$ with compact support in $[-2T,2T]$, then its restriction to $[-T,T]$ is a (local) solution of KdV.
We thus want to show that the map $\Gamma$ is a contraction on the space
\begin{equation}\label{contmap}
K=\left\{u\in Y^s\cap Y^0:\|u\|_{ Y^0}\leq C_1 \|u_0\|_{L^2} \mbox{ and } \|u\|_{Y^s}\leq C_1\|u_0\|_{H^s} \right\}
\end{equation}
with constant $C_1$ big enough. The topology we use for the contraction argument is $Y^0$.

First, we can see that the ball of radius $R'$ in $Y^s$ is closed in $Y^0$ as we can describe the two norms involved in $Y^s$ in the following way. We have
$$
\|u\|_{X^{s,1/2}} = \sup \left\{ \sum_k \int d\tau g(k,\tau) \hat u(k,\tau) \Big| \||k|^{-s}\langle \tau - k^3\rangle^{-1/2} g(k,\tau)\|_{l^2_k,L^1_\tau} \leq 1 \right\}
$$
and
$$
\||k|^s\hat u(k,\tau)\|_{l^2_k,L^1_\tau} = \sup  \left\{ \sum_k \int d\tau g(k,\tau) \hat u(k,\tau) \Big| \||k|^{-s} g(k,\tau)\|_{l^2_k,L^\infty_\tau} \leq 1 \right\} \; .
$$
Perform the usual argument to get that the balls of $Y^s$ are closed in $Y^0$. Therefore $K$ is closed in $Y^0$.

By Propositions \ref{linp} and \ref{bilp} we have for all $s\geq0$
\begin{equation*}
\|\Gamma(u)\|_{Y^s}\leq C\|u_0\|_{H^s}+CT^\frac1{12}\|u\|_{Y^s}\|u\|_{Y^0} 
\end{equation*}
hence using it with $s=0$ we have
$$
\|\Gamma(u)\|_{Y^0}\leq C\|u_0\|_{L^2}+CT^\frac1{12}C_1^2 \|u_0\|_{L^2}^2
$$
and with general $s\geq 0$
$$
\|\Gamma(u)\|_{Y^s}\leq C\|u_0\|_{H^s}+CT^\frac1{12}C_1^2\|u_0\|_{L^2}\|u_0\|_{H^s} 
$$
so that for times of order $\|u_0\|_{L^2}^{-12}$, $\Gamma$ maps $K$ to $K$. We now prove that $\Gamma$ is a contraction on $K$, meaning that there exists $\theta\in (0,1)$ such that
\begin{equation*}
\|\Gamma(u)-\Gamma(v)\|_{Y^0}\leq\theta\|u-v\|_{Y^0}
\end{equation*}
for all $u,v\in K$. Since $u^2-v^2=(u+v)(u-v)$, the application of Propositions \ref{linp} and \ref{bilp} yields again, for every $s\geq0$,
\begin{equation*}
\|\Gamma(u)-\Gamma(v)\|_{Y^0}\leq cT^{\frac1{12}}\Big[\|u-v\|_{Y^0}(\|u\|_{Y^0}+\|v\|_{Y^0})\Big].
\end{equation*}
Choosing $T\leq C \|u_0\|_{L^2}^{-12}$ with $C$ small enough, since in $K$, $(\|u\|_{Y^0}+\|v\|_{Y^0}) \leq 2C_1 \|u_0\|_{L^2}$, yields
\begin{equation*}
\|\Gamma(u)-\Gamma(v)\|_{Y^s}\leq \theta \|u-v\|_{Y^s} \; .
\end{equation*}

To prove the uniqueness, we proceed as follows. Let $u$ be a solution of KdV on the time interval $[-T,T]$. Then we can use the map $\tilde \Gamma_u$ from $Y^s$ to itself
defined as
$$
\tilde{\Gamma}_u(\tilde{u})(t,x) = u(t,x)
$$
if $t \in [-T,T]$ and
\begin{eqnarray*}
\tilde{\Gamma}_u(\tilde{u})(T+t,x) &=& \eta_T(T+t)S(T+t)u_0-\eta_T\int_0^TS(T+t-t')(\eta_T(t')\frac12\partial_x(u)^2)dt'\\
& & -\eta_T(T+t)\int_T^{T+t}S(T+t-t')\frac12\partial_x(\tilde{u}^2)dt'
\end{eqnarray*}
and the analogous with $-T$ to extend $u$ into $\tilde u$ a fixed point of $\Gamma$. The uniqueness of the the fixed point of $\Gamma$ thus yields the uniqueness of the solution.

Besides, we have that if $u'$ is the fixed point of $\Gamma$ and $u$ the solution,
$$
\|1_{t \in [-T,T]}u\|_{Y^s} \leq \|u'\|_{Y^s} \lesssim \|u_0\|_{H^s}\; .
$$

\end{proof}

\begin{remark}
Notice that since the lifespan $T$ of the solution only depends on the $L^2$ norm of the initial datum and since the $L^2$ norm is conserved under the KdV flow, such a solution is global.
\end{remark}

\subsection{Approximation of KdV and local uniform convergence}

To conclude this section we prove a result of uniform convergence of the solutions of a projected KdV problem to the solution of \eqref{pb}, which will be useful in the proof of the invariance of the measure built in Section 4.

Let us denote by $E_N$, $N\geq1$, the vector spaces given by
\begin{equation*}
E_N={\rm Span}\left\{c_n,s_n| 1\leq n\leq N\right\}
\end{equation*}
where
\begin{equation*}
c_n(x)=\frac{\cos(nx)}{\sqrt{\pi}},\qquad s_n(x)=\frac{\sin(nx)}{\sqrt{\pi}},
\end{equation*}
$n\geq 1$, is the usual orthonormal basis of $L^2$ functions with $0$ mean value, and let $\Pi_N$ be the orthogonal projector on $E_N$. Let us moreover consider the projected KdV
\begin{equation}
\begin{cases}
\partial_tu+\partial_x^3u+\Pi_N\left(\frac12\partial_x\left(\Pi_Nu\right)^2\right)=0, \quad x\in\T,N\geq 1\\
u(x,0)=u_0(x)
\end{cases}
\end{equation}
and denote with $\Psi_N(t)u_0$ the corresponding solution. Note that the equation can be divided into an ODE on $E_N$ : 
$$
\partial_tu+\partial_x^3u+\Pi_N\left(\frac12\partial_x u^2\right)=0
$$
with initial datum $\Pi_N u $, whose non linearity is Lipschitz and with a $L^2$ invariance, hence it is globally well-posed; and a linear equation on the orthogonal of $E_N$ : 
$$
\partial_tu+\partial_x^3u=0
$$
with initial datum $(1-\Pi_N)u$, also globally well-posed. The same methods as in the local well-posedness of KdV gives the same bounds on $\Psi(t) u_0$ and $\Psi_N(t) u_0$. Hence, we have
$$
\|\Psi_N(t)u_0 \|_{Y^s} \leq C \|u_0\|_{H^s}
$$
where we cut off the times outside $[-T,T]$, where $T$ is of order $\|u_0\|_{L^2}^{-12}$ (notice that the constants do not depend on $N$).

Then we can prove the following proposition, that we will refer in the sequel local uniform convergence.

\begin{proposition}\label{luc}
Let $R' \geq 0$, $\sigma > s\geq 0$. There exists a time $T>0$ such that for every $R,\varepsilon>0$, there exists $N_\varepsilon$ such that for every $N>N_\varepsilon$, every $t\in[-T,T]$ and every $u_0$ such that $\|u_0\|_{H^\sigma}\leq R$,
 and $\|u_0\|_{L^2}\leq R'$,\begin{equation*}
\|\Psi_N(t)u_0-\Psi(t)u_0\|_{H^s}\leq \varepsilon.
\end{equation*}
\end{proposition}

\begin{proof}
We use the notation $u_N = \Psi_N(t) u_0$.  We thus can write
$$
u_N = S(t) u_0 - \frac12 \int_{0}^t S(t-t')\Pi_N \partial_x ((\Pi_N u_N(t'))^2) dt' \; .
$$
We rely on the facts that the $H^s$ norm of $\| u(t) -u_N(t)\|_{H^s}$ for $t\in[-T,T]$ is bounded by the $Y^s$ norm of $\eta_T (t) (u-u_N)$ and that the $Y^s$ norms of $u$ and $u_N$ are bounded by the $H^s$ norm of $u_0$ if we cut off the times outside $[-T,T]$.

From this, we get the bound
\begin{equation*}
\|\Psi_N(t)u_0-\Psi(t)u_0\|_{Y^s}\leq\left\|\frac12\int S(t-t')\left(\Pi_N\partial_x\left(\Pi_Nu_N\right)^2-\partial_xu^2\right)(t')dt'\right\|_{Y^s}.
\end{equation*}
We will estimate this term dividing it into the sum
\begin{equation*}
\left\|\frac12\int S(t-t')\left((1-\Pi_N)\partial_x\left(\Pi_Nu_N\right)^2\right)(t')dt'\right\|_{Y^s}+
\left\|\int S(t-t')\left(\partial_x\left((\Pi_Nu_N)^2-u^2\right)\right)(t')dt'\right\|_{Y^s}
\end{equation*}
\begin{equation*}
=I+II.
\end{equation*}
We need the following property that holds for every $N\geq1$ and every $\sigma>s$:
\begin{equation}\label{proj}
\|\left(1-\Pi_N\right)F\|_{Y^s}\leq N^{s-\sigma}\|F\|_{Y^\sigma}.
\end{equation}
This can be proved using the same argument as in the Sobolev spaces, for every $\sigma>s$,
\begin{equation*}
\sum_{|k|>N}|k|^{2s}|f|^2=
\sum_{|k|>N}|k|^{2\sigma}|k|^{2(s-\sigma)}|f_k|^2\leq N^{2(s-\sigma)}\sum_k|k|^{2\sigma}|f_k|^2.
\end{equation*} 
Since $\Pi_N$ commutes with $S(t)$ we thus have
\begin{equation*}
I\leq\frac12N^{s-\sigma}\left\|\int S(t-t')\partial_x(\Pi_Nu_N)^2(t')dt'\right\|_{Y^\sigma}.
\end{equation*}
We now exploit our linear and bilinear estimates in the form
\begin{equation*}
\left\|\eta_T(t)\int_0^tS(t-t')\partial_x(u^2)(t')dt'\right\|_{Y^\sigma}\leq C T^{1/{12}}\|u\|_{Y^0}\|u\|_{Y^\sigma}
\end{equation*} 
to have
\begin{equation*}
N^{s-\sigma}\left\|\int S(t-t')\partial_x(\Pi_Nu_N)^2(t')dt'\right\|_{Y^\sigma}\leq
T^{1/12}N^{s-\sigma}\left\|\Pi_Nu_N\right\|_{Y^\sigma}\left\|\Pi_Nu_N\right\|_{Y^0}
\end{equation*}
and then we apply that $\|\Pi_N u_N\|_{Y^s}$ is less than $\|u_N\|_{Y^s}$ which is bounded by the $H^s$ norm of $u_0$. Therefore,
$$
I \leq T^{1/12} N^{s-\sigma}RR'\; .
$$
Analogously, for $II$ we have
\begin{equation*}
II\lesssim T^{1/12}\|\Pi_Nu_N-u\|_{Y^s}\big(\|\Pi_Nu_N\|_{Y^0}+\|u\|_{Y^0}\big) + T^{1/12}\|\Pi_Nu_N-u\|_{Y^0}\big(\|\Pi_Nu_N\|_{Y^s}+\|u\|_{Y^s}\big).
\end{equation*}
Writing now
\begin{equation*}
\|\Pi_Nu_N-u\|_{Y^s}\leq\|\Pi_Nu_N-u_N\|_{Y^s}+\|u_N-u\|_{Y^s}
\end{equation*}
and applying \eqref{proj} yields (since $s<\sigma$)
\begin{equation*}
II\lesssim T^{1/12}R
\left(N^{s-\sigma}R'+\|u_N-u\|_{Y^s} + R' \|\Pi_Nu_N-u\|_{Y^0}\right)\; .
\end{equation*}
Taking $s = 0$, $R = R'$ and $T$ of order $R^{-12}$ small enough, we get
$$
\|u-u_N\|_{Y^0} \leq \frac{1}{2} \|u-u_N\|_{Y^0} + CR N^{-\sigma}
$$
hence the local uniform convergence holds in $Y^0$. Knowing that we input this information into the estimate in $Y^s$, that is, for $T$ of order $R^{-12}$, we have
$$
\|u-u_N\|_{Y^s} \leq \frac12 \|u-u_N\|_{Y^s} + CRR' N^{s-\sigma}  + CR' \|u-u_N\|_{Y^0}
$$
and then we get the local uniform convergence in $Y^s$.
\end{proof}

\begin{remark}Note that the time of local uniform convergence depends only on the $L^2$ norm of the initial datum, whereas the rate of convergence depends on both the $L^2$ and $H^\sigma$ norm of the initial datum.\end{remark}

\section{Continuity of the flow}\label{sec-contflow}
 
\subsection{Global estimates}

We prove here some global bounds on KdV iterating the estimates provided by the local well-posedness proposition.

\begin{lemma}There exists $C$ such that for all times $t\in \R$, all $u_0$ and $v_0$ in $H^s$, we have the estimates 
$$
\|\Psi(t)u_0 - \Psi(t)v_0\|_{H^s}\leq C \left( \|u_0-v_0\|_{H^s} + (\|u_0\|_{H^s} + \|v_0\|_{H^s}) \|u_0-v_0\|_{L^2}\right) e^{c|t| (\|u_0\|_{L^2}+\|v_0\|_{L^2})^{12}}\; ,
$$
$$
\|\Psi(t)u_0 - \Psi(t)v_0\|_{H^s}\leq C  \|u_0-v_0\|_{L^2} e^{c|t| (\|u_0\|_{L^2}+\|v_0\|_{L^2})^{12}}\; .
$$
\end{lemma}

\begin{proof} Let $T = \frac{1}{C(\|u_0\|_{L^2}+\|v_0\|_{L^2})^{12}}$ the time of local existence and uniqueness of the solutions. We have that, thanks to local well posedness in $Y^s$
$$
\|\Psi(T)u_0 \|_{H^s} \leq C\|u_0 \|_{H^s}\; .
$$
Let $T_n = nT$. Since the $L^2$ norm is invariant under the flow of KdV, we get by induction
$$
\|\Psi(T_n)u_0 \|_{H^s} \leq C^{|n|}\|u_0 \|_{H^s}\; .
$$
Then, we use that, thanks to the bilinear estimates, for local times,
$$
\|\Psi(T)u_0 - \Psi(T)v_0\|_{L^2} \leq C \|u_0 - v_0\|_{L^2}
$$
and
$$
\|\Psi(T)u_0 - \Psi(T)v_0\|_{H^s} \leq C T^{1/12}(\|u_0\|_{H^s} + \|v_0\|_{H^s}) \|u_0-v_0\|_{L^2} + C  \|u_0-v_0\|_{H^s} \; .
$$
Thanks to the conservation of the $L^2$ norm, we get by induction
$$
\|\Psi(T_n)u_0 - \Psi(T_n)v_0\|_{L^2} \leq C^{|n|} \|u_0 - v_0\|_{L^2}\; .
$$
Then, we rewrite, substituting $u_0$ by $\Psi(T_n)u_0$,
\begin{eqnarray*}
\|\Psi(T_{n+1})u_0 - \Psi(T_{n+1})v_0\|_{H^s} &\leq &C T^{1/12}(\|\Psi(T_n)u_0\|_{H^s} + \|\Psi(T_n)v_0\|_{H^s}) \|\Psi(T_n)u_0-\Psi(T_n)v_0\|_{L^2}\\
 & & + C  \|\Psi(T_n)u_0-\Psi(T_n)v_0\|_{H^s} \; .
 \end{eqnarray*}
We plug into this inequality the estimates on $\|\Psi(T_n)u_0\|_{H^s}$ and $\|\Psi(T_n)u_0 - \Psi(T_n)v_0\|_{L^2}$ to get
$$
\|\Psi(T_{n+1})u_0 - \Psi(T_{n+1})v_0\|_{H^s} \leq C T^{1/12}C^{2|n|}(\|u_0\|_{H^s} + \|v_0\|_{H^s}) \|u_0-v_0\|_{L^2} + C  \|\Psi(T_n)u_0-\Psi(T_n)v_0\|_{H^s}\; .
$$
With $D$ big enough, we get by induction
$$
\|\Psi(T_{n})u_0 - \Psi(T_{n})v_0\|_{H^s} \leq D^{|n|} \left( (\|u_0\|_{H^s} + \|v_0\|_{H^s}) \|u_0-v_0\|_{L^2} + \|u_0-v_0\|_{H^s}\right)\; .
$$
As $n$ is equal to $CT_n (\|u_0\|_{L^2} + \|v_0\|_{L^2})^{12}$, we get the result at the discrete times $T_n$, the local properties extend it to all times.
\end{proof}

\subsection{Continuity with regard to the Wasserstein metrics}

We now state the main result of this paper, proving local Lipschitz continuity of the KdV flow with respect to the Wasserstein metrics.

\begin{theorem}\label{th1}
If $\Psi(t)$ is the KdV flow on $L^2$ and if $\mu$ and $\nu$ belong to $M_{s,p}$ then $\mu^t$ and $\nu^t$ belong to $M_{s,p}$ and
\begin{equation}\label{thtes}
\|\mu^t-\nu^t\|_{s,p}\leq C(1+ R_2) e^{c|t|R_1}
\|\mu-\nu\|_{s,p}
\end{equation}
where $R_1=(\|u\|_{L^\infty_\mu,L^2}+\|v\|_{L^\infty_\nu,L^2})^{12}$ and $R_2=\|u\|_{L^p_\mu,H^s}+\|v\|_{L^p_\nu,H^s}$.

\end{theorem}

\begin{proof}
We assume that $\mu \in M_{s,p}$. We recall that $\mu^t(A) = \mu(\Psi(t)^{-1}A)$ for all measurable set $A$. We get that
$$
\mu^t \left( \left\{ u \Big| \|u\|_{L^2} > \lambda \right\}  \right) = \mu \left( \left\{ u \Big| \; \|\Psi(t)u\|_{L^2}> \lambda\right\} \right)
$$
thus, thanks to the conservation of the $L^2$ norm, and with the assumptions on $\mu$, $u$ belongs to $L^\infty_{\mu^t},L^2$.

We also have thanks to a change of variable that
$$
\|u\|_{L^p_{\mu^t},H^s} = \left( \int \|\Psi(t) u_0\|_{H^s}^p d\mu(u) \right)^{1/p} \; .
$$
The global bounds give an estimate of $\|\Psi(t) u_0\|_{H^s}$
$$
\|\Psi(t) u_0\|_{H^s} \leq e^{c |t| \|u_0\|_{L^2}^{12}}\|u_0\|_{H^s}
$$
which yields
$$
\|u\|_{L^p_{\mu^t},H^s}  \leq e^{c |t| \|u_0\|_{L^\infty_\mu,L^2}^{12}}\|u_0\|_{L^p_\mu,H^s}
$$
hence $\mu^t$ belongs to $M_{s,p}$.

Let $\gamma \in \textrm{Marg}(\mu,\nu)$ a measure on $H^s\times H^s$ with marginals $\mu$ and $\nu$. We call $\gamma^t$ the measure defined on the generating measurable sets of $H^s\times H^s$ as
$$
\gamma^t(X\times Y) = \gamma(\Psi(t)^{-1}X\times \Psi(t)^{-1}Y)\; .
$$
We remark that
$$
\gamma^t(X \times H^s) = \gamma(\Psi(t)^{-1}X\times H^s)
$$
even if $\Psi(t)^{-1}H^s$ may be larger than $H^s$, $\gamma$ is supported in $H^s\times H^s$, and since its marginals are $\mu$ and $\nu$,
$$
\gamma^t(X \times H^s) = \mu(\Psi(t)^{-1}X) = \mu^t(X)\; .
$$
The symmetrical argument yields $\gamma^t \in \textrm{Marg}(\mu^t,\nu^t)$. Therefore, we get that, by definition of the Wasserstein metrics, for every $\gamma$ in Marg($\mu,\nu$),
$$
W_{0,\infty}(\mu^t,\nu^t) \leq \big\|\; \|u-v\|_{L^2} \big\|_{L^\infty_{\gamma^t}}\; .
$$
We perform the change of variable $u = \Psi(t) u_0$ to get
$$
W_{0,\infty}(\mu^t,\nu^t) \leq \big\|\; \|\Psi(t)u_0-\Psi(t)v_0\|_{L^2} \big\|_{L^\infty_{\gamma}}\; .
$$
We use the global bound
$$
\|\Psi(t)u_0-\Psi(t)v_0\|_{L^2} \leq e^{c|t|(\|u_0\|_{L^2}+\|v_0\|_{L^2})^{12}} \|u_0-v_0\|_{L^2}
$$
and remark that $\gamma$ almost surely $(\|u_0\|_{L^2}+\|v_0\|_{L^2})^{12}$ is less than $R_1$, hence
$$
W_{0,\infty}(\mu^t,\nu^t) \leq e^{c|t|R_1}\big\|\; \|u_0-v_0\|_{L^2} \big\|_{L^\infty_{\gamma}}
$$
and taking the infimum over $\gamma \in \textrm{Marg}(\mu,\nu)$ gives
$$
W_{0,\infty}(\mu^t,\nu^t) \leq e^{c|t|R_1} W_{0,\infty}(\mu,\nu) \leq e^{c|t|R_1} \|\mu-\nu\|_{s,p}.
$$

With the same remarks on $\gamma^t$ and performing the same change on variable, we get
$$
W_{s,p}(\mu^t,\nu^t)^p \leq \int \|\Psi(t)u_0 - \Psi(t)v_0\|_{H^s}^pd\gamma(u_0,v_0) \; .
$$
We use the global bound, $\gamma$-almost surely
$$
\|\Psi(t)u_0 - \Psi(t)v_0\|_{H^s} \leq Ce^{c|t|R_1} \left( \|u_0-v_0\|_{L^2}(\|u_0\|_{H^s}+\|v_0\|_{H^s}) + \|u_0-v_0\|_{H^s}\right)
$$
to get as $\|u_0\|_{L^p_\gamma,H^s} = \|u_0\|_{L^p_\mu,H^s}$ since $\gamma$ has for marginals $\mu$ and $\nu$,
$$
W_{s,p}(\mu^t,\nu^t)^p \leq Ce^{cp|t|R_1} \left( \|u_0-v_0\|_{L^\infty_\gamma,L^2}^pR_2^p + \int \|u_0-v_0\|_{H^s}^p d\gamma(u_0,v_0) \right)\; .
$$
Taking the infimum over $\gamma \in \textrm{Marg}(\mu,\nu)$ yields
$$
W_{s,p}(\mu^t,\nu^t) \leq Ce^{c|t|R_1} \left( W_{0,\infty}(\mu,\nu) R_2 + W_{s,p}(\mu,\nu) \right)
$$
which concludes the proof.
\end{proof}

\section{Construction of an invariant measure for KdV}\label{sec-cons}

\subsection{Construction of measures, and notations}

We recall the construction of an invariant measure for KdV introduced by Bourgain in \cite{Binvkdv}

Let us first build an invariant measure $\mu$ through the linear flow of KdV. 

Let $(\Omega,\mathcal A,\mathbb P)$ be a probability space such that there exist two sequences $(h_n)_{n\geq 1} $ and $(l_n)_{n\geq 1}$ of independent real centered and normalized Gaussian variables. For all integers $0\leq N< M$, we write $\phi_N^M$ the random variable 
$$
\phi_N^M (\omega,x)= \sum_{n=N+1}^M \frac{h_n(\omega)c_n(x)+l_n s_n(x)}{n}\; ,
$$
where the set
$$
\lbrace c_n(x) = \frac{\cos(nx)}{\sqrt \pi} \; ,\; s_n(x) = \frac{\sin(nx)}{\sqrt \pi} \; \Big| \; n\geq 1\rbrace
$$
is the usual orthonormal basis of the functions of $L^2$ with $0$ mean value.

\begin{lemma} For all $s<\frac{1}{2}$, the sequence $(\phi_N^M)_{M>N}$ converges in $L^2(\Omega,H^s)$ when $M\rightarrow \infty$. \end{lemma}

\begin{proof} The sequence is Cauchy. Indeed, for $M_1< M_2$, we have
$$
\|\phi_N^{M_2} - \phi_N^{M_1}\|_{L^2(\Omega,H^s)} = E(\|\phi_{M_1}^{M_2}\|^2_{H^s})^{1/2}
$$
where $E$ is the expectation with regard to $\mathbb P$. Thus,
$$
\|\phi_N^{M_2} - \phi_N^{M_1}\|_{L^2(\Omega,H^s)} = \left( E\Big( \sum_{n=M_1+1}^{M_2}\frac{|h_n|^2+|l_n|^2}{n^{2(1-s)}}\Big) \right)^{1/2}
$$
and since $E(|h_n|^2) = E(|l_n|^2) = 1$,
$$
\|\phi_N^{M_2} - \phi_N^{M_1}\|_{L^2(\Omega,H^s)} = \left( \sum_{n=M_1+1}^{M_2}\frac{2}{n^{2(1-s)}}\right)^{1/2}\; .
$$
Since $s<1/2$, the series of general term $\frac{1}{n^{2(1-s)}}$ converges and hence the sequence is Cauchy. As $L^2(\Omega,H^s)$ is complete, the sequence $(\phi_N^M)_M$ converges towards a limit $\phi_N$.\end{proof}

Let now $E_N$ be the vector space spanned by the set
$$
\lbrace c_n,s_n\;|\; 1\leq n\leq N\rbrace \; ,
$$
$E_N^\bot$ its orthogonal complement in $L^2$, and $\Pi_N$ the orthogonal projection on $E_N$.

We build measures on the $\sigma$-algebra associated to $H^{1/2-} = \bigcap_{s<1/2} H^s$. The topology of $H^{1/2-}$ is generated by the union of the topologies of $H^s$ traced on $H^{1/2-}$. In other words, an open set of $H^{1/2-}$ is a reunion of open sets of $H^s$ intersected with $H^{1/2-}$. 

We call $\mu_N^M$ the measure on the topological $\sigma$-algebra of $H^{1/2-}$ traced on $E_N^\bot$ induced by $\phi_N^M$ that is 
$$
\mu_N^M (A) = \mathbb P \Big((\phi_N^M)^{-1}(A)\Big)\; ,
$$
and $\mu_N$ the one induced by $\phi_N$. By convention, we write $\mu^M = \mu_0^M$ and $\mu = \mu_0$. As the support of $\mu^N$ is included in $E_N$, we have that for all $N$
$$
\mu = \mu^N\otimes \mu_N\; .
$$

\begin{remark}Notice that this is an abuse of notation since $\mu^N$ is a measure defined on the whole space $L^2$, but as its support is included in $E_N$, we denote by $\mu^N$ both the measure on $L^2$ and its restriction to $E_N$.\end{remark}

We now define an invariant measure $\rho$ for the KdV equation and invariant measures $\rho_N$ under the approximation in finite dimension of KdV given in the last section.

We introduce the following map on $H^{1/2-}$ :
$$
f(u) = \chi(\|u\|_{L^2}) e^{\int u^3}
$$
where $\chi (x) =1$ if $x\leq 1$ and $\chi(x)= 0$ otherwise, and $f_N$ the map given by
$$
f_N(u) = \chi(\|u\|_{L^2})e^{\int(\Pi_N(u))^3}\; .
$$

In order to build the invariant measure by the flow of KdV, we need to prove that $f$ is in $L^1_\mu$. The fact that $f_N$ is also in $L^1_\mu$ and that the sequence $(f_N)_N$ converges towards $f$ in $L^1_\mu$ is an element of the proof of the invariance of the measure. We sum up these properties in Proposition \ref{prop-cgfn} because the proofs require the same techniques.

\begin{proposition}The support of $\mu$ is included in $L^\infty(\T)$. Besides there exist $C,c> 0$ such that for all $R$
$$
\mu (\|u\|_{L^\infty} > R) \leq Ce^{-cR^2}
$$
and for all $N$ and all $R$  
$$
\mu^N (\|u\|_{L^\infty} > R) \leq Ce^{-cR^2}\; .
$$
\end{proposition}

\begin{proof}Let $s<1/2$ and $p>\frac{2}{s}$ such that the Sobolev embedding $ W^{s,p}\subset L^\infty$ holds. Let us prove that the support of $\mu$ is included in $W^{s,p}$. To do that, we bound the $L^q_\mu$ norm of $\|u\|_{W^{s,p}}$ for any $q\geq p$. After a change of variable, we get that
$$
\|u\|_{L^q_\mu, W^{s,p}} =\|\phi_0\|_{L^q_{\mathbb P},W^{s,p}}
$$
and by Minkowski inequality, since $q\geq p$
$$
\|u\|_{L^q_\mu, W^{s,p}} \leq \|\sum_{n\geq 1} \frac{h_nc_n(x)+l_ns_n(x)}{n^{1-s}}\|_{L^p(\T, L^q(\Omega))}\; .
$$
Since for all $x\in \T$,
$$
\sum_{n\geq 1} \frac{h_nc_n(x)+l_ns_n(x)}{n^{1-s}}
$$
is a centered Gaussian variable in $\R$, we have that
$$
\|\sum_{n\geq 1} \frac{h_nc_n(x)+l_ns_n(x)}{n^{1-s}}\|_{L^q(\Omega)} \leq C\sqrt q \|\sum_{n\geq 1} \frac{h_nc_n(x)+l_ns_n(x)}{n^{1-s}}\|_{L^2(\Omega)}\; ,
$$
as for all centered Gaussian variable and all $q \geq 1$, we have, by induction over $q \in \N$ and interpolation, $E(|Z|^q)^{1/q} \leq C\sqrt q E(Z^2)^{1/2}$. We refer to \cite{BTranI} for the proof.

Hence, as the $h_n$ and the $l_n$ are all independent from each other, we get that 
$$
\|\sum_{n\geq 1} \frac{h_nc_n(x)+l_ns_n(x)}{n^{1-s}}\|_{L^p(\Omega)} \leq C\sqrt p \left( \sum_{n\geq 1}\frac{|c_n(x)|^2+|s_n(x)|^2}{n^{2(1-s)}}\right)^{1/2}\; .
$$
Using once more the fact that $s<1/2$ so that the series $\sum n^{2(s-1)}$ converges and that $|c_n(x)|^2+|s_n(x)|^2 = \pi^{-1}$, we get that the $L^q(\Omega)$ norm of
$$
\sum_{n\geq 1} \frac{h_nc_n(x)+l_ns_n(x)}{n^{1-s}}
$$
is bounded by a a constant independent from $x$. Then, taking the $L^p(\T)$ norm, we have that the $L^q_\mu,W^{s,p}$ norm of $u$ is finite. Therefore, $u$ is $\mu$-almost surely in $W^{s,p}$ and hence $\mu$-almost surely in $L^\infty$. Replacing $\phi_0$ by $\phi_0^N$ and following the same proof ensures that there exists a constant $C$ such that for all $q\geq p$ and all $N$  
$$
\|u\|_{L^q_{\mu^N},L^\infty} \leq \|u\|_{L^q_{\mu^N},W^{s,p}} \leq C \sqrt q \; .
$$
Then, for all $q\geq p$, we can bound the probabilities
$$
\mu (\|u\|_{L^\infty} \geq R) \mbox{ and } \mu_N (\|u\|_{L^\infty} \geq R) \; .
$$
Indeed, by Markov's inequality,
$$
\mu (\|u\|_{L^\infty} \geq R) = \mu(\|u\|_{L^\infty}^q \geq R^q) \leq R^{-q} E_\mu (\|u\|_{L^\infty}^q)\; ,
$$
hence
$$
\mu (\|u\|_{L^\infty} \geq R)\leq R^{-q} C^q q^{q/2}
$$
and we have the same inequality with $\mu_N$. If $R\geq Ce \sqrt p$, then $\left( \frac{R}{Ce}\right)^2 \geq p$, thus we can choose $q = \left( \frac{R}{Ce}\right)^2$ in the previous inequality, yielding
$$
\mu (\|u\|_{L^\infty} \geq R)\leq e^{-q} = e^{-c R^2}
$$
with $c = \frac{1}{Ce}$. If $R\leq Ce\sqrt p$, then
$$
\mu (\|u\|_{L^\infty} \geq R)e^{c R^2} \leq C_0 = e^{c C^2e^2p}\; ,
$$
therefore for all $R$,
$$
\mu (\|u\|_{L^\infty} \geq R)\leq  C_0 e^{-c R^2}\; .
$$
\end{proof}

\begin{remark}\label{HSldge} By replacing the $L^\infty$ and $W^{s,p}$ norms in the proof by the $H^s$ norms, we have that for all $s<1/2$, there exist $C_s,c_s >0$ such that for all $R$,
$$
\mu (\|u\|_{H^s}\geq R) \leq C_s e^{-c_s R^2}\; .
$$\end{remark}

\begin{proposition}\label{prop-cgfn}The maps $f$ and $f_N$ are in $L^1_\mu$ and the sequence $(f_N)_{N}$ converges toward $f$ in $L^1_\mu$.\end{proposition}

\begin{proof}As we have the inequality
$$
\int{u^3} \leq \|u\|_{L^2}^2 \|u\|_{L^\infty}
$$
we get that, since $\|u\|_{L^2}\leq 1$ in the support of $f$,
$$
|f(u)| \leq e^{\|u\|_{L^\infty}} \; .
$$
Besides, since $\|\Pi_N u\|_{L^2} \leq \|u\|_{L^2}$ in the support of $f_N$,
$$
|f_N(u)| \leq e^{\|\Pi_N u\|_{L^\infty}} \; .
$$
We also have, since $\mu = \mu^N\otimes \mu_N$  
$$
\int  e^{\|\Pi_N u\|_{L^\infty}} d\mu(u) = \int  e^{\| u\|_{L^\infty}}d\mu^N(u)\; .
$$
Hence, the problem is reduced to prove that $e^{\|u\|_{L^\infty}}$ is $\mu$ and $\mu^N$ integrable, which comes from the fact that
$$
\mu(\|u\|_{L^\infty} \geq R)\leq  C_0 e^{-c R^2}\mbox{ and }\mu_N(\|u\|_{L^\infty} \geq R)\leq  C_0 e^{-c R^2}\; .
$$
Besides, as $C_0$ and $c$ are independent from $N$, the $L^1_\mu$ norm of $f_N$ is bounded uniformly in $N$.

Let us now prove the convergence of the sequence. We compare $f(u)$ and $f_N(u)$. For every $u$, we have  
$$
|f(u)-f_N(u)|\leq \Big| \int (u^3 - (\Pi_N u)^3)\Big| f(u)f_N(u)
$$
using the inequality $|e^x-e^y|\leq |x-y|e^{|x|} e^{|y|}$ and that $\chi = \chi^2$. Since $|u^3 - (\Pi_N u)^3|$ is bounded by $C|u-\Pi_N u| (u^2 +(\Pi_N u)^2)$, we get, by applying a H\"older inequality
$$
\int |u^3 - (\Pi_N u)^3| \leq C \|(1-\Pi_N)u\|_{L^2} \left( \|u\|_{L^2}+\|\Pi_N u\|_{L^2} \right) \left( \|u\|_{L^\infty}+\|\Pi_N u\|_{L^\infty} \right) \; .
$$
We then use that on the support of $f$ and $f_N$, the $L^2$ norms of $u$ and $\Pi_N u$ are bounded by $1$ and that $u$ belongs to $H^s$ on the support of $\mu$ as long as $s< 1/2$ to get 
$$
|f(u)-f_N(u)|\leq CN^{-s} \|u\|_{H^s} \left( \|u\|_{L^\infty}+\|\Pi_N u\|_{L^\infty} \right) f(u)f_N(u) \; .
$$
Finally, using Remark \ref{HSldge},
$$
\mu(\|u\|_{H^s} \geq R) \leq C_se^{-c_s R^2} \; , \; \mu(\|u\|_{L^\infty}\geq R) \leq Ce^{-cR^2} \; ,
$$
we get that
$$
\|u\|_{H^s} (\|u\|_{L^\infty}+\|\Pi_N u\|_{L^\infty}) e^{\|u\|_{L^\infty}+\|\Pi_N u\|_{L^\infty}}
$$
is $\mu$-integrable with a $L^1_\mu$ norm independent from $N$. Therefore,
$$
\|f-f_N\|_{L^1_\mu} \leq CN^{-s} \; ,
$$
which concludes the proof.
\end{proof}

We write then $\kappa = \|f\|_{L^1_\mu}^{-1}$ and we denote by $\rho_N$ and $\rho$ the measures defined as  
$$
d\rho_N(u) = \kappa f_N(u) d\mu(u) \; , \; d\rho(u) = \kappa f(u) d\mu(u) \; .
$$
The measure $\rho_N$ is built to be invariant under the flow of the approximation of KdV, which will lead, thanks to convergence properties to the invariance of $\rho$ under the flow of KdV.

\subsection{Invariance by the linear flow}

In this subsection, we prove the invariance of $\mu_N$ under the linear flow $S(t)$ of the equation $\partial_t+\partial_x^3 = 0$.

\begin{proposition}\label{prop-openset} If $U$ is an open set of $H^{1/2-}$, then, for all $N$,
\begin{equation}\label{openset}
\mu_N (U) \leq \liminf_{M\rightarrow \infty} \mu_N^M (U)\; 
\end{equation}
\end{proposition}

We deduce from this proposition the following corollary.

\begin{corollary}\label{cor-closet} For every closed set $F$ of $H^{1/2-}$ and every $N$  
$$
\mu_N (F) \geq \limsup_{M\rightarrow \infty} \mu_N^M(F) \; .
$$
\end{corollary}

\begin{proof} Proposition \ref{prop-openset}. We can reduce the problem by proving it on the intersections of open sets of $H^s$ with $H^{1/2-}$ for every $s< 1/2$. Indeed, if $U$ is an open set of $H^{1/2-}$ it can be described as the union
$$
U = \Big(\bigcup_{s<1/2} U_s\Big)
$$
with $U_s$ the intersection of an open set in $H^s$ and $H^{1/2-}$. We thus can write $U$ as the limit of an increasing sequence of sets 
$$
U= \bigcup_{s_0<1/2} \bigcup_{s\leq s_0} U_s 
$$
and hence
$$
\mu_N (U) = \sup_{s_0} \mu_N \Big(\bigcup_{s\leq s_0} U_s\Big) \; ,\; \mu_N^M (U) = \sup_{s_0} \mu_N^M \Big(\bigcup_{s\leq s_0} U_s\Big)\; .
$$
Then, if the proposition is true for every intersection of an open set of $H^{s_0}$ with $H^{1/2-}$, for any $s_0 < 1/2$, as $\bigcup_{s\leq s_0} U_s$ is one of these sets, we get 
$$
\mu_N(U) \leq \sup_{s_0 <1/2} \liminf_{M \rightarrow \infty} \mu_N^M \Big(\bigcup_{s\leq s_0} U_s\Big) \; .
$$
We exploit the property that for any map $f$ into a totally ordered set 
$$
\sup_{s_0} \liminf_{M} f(s_0,M) \leq \liminf_M \sup_{s_0}f(s_0,M)
$$
to conclude. We have that
$$
\mu_N(U)\leq \liminf_{M \rightarrow \infty} \sup_{s_0 <1/2} \mu_N^M \Big(\bigcup_{s\leq s_0} U_s\Big)
$$
and 
$$
\sup_{s_0 <1/2} \mu_N^M \Big(\bigcup_{s\leq s_0} U_s\Big) = \mu_N^M (U)\; .
$$

Let $V$ be an open set of $H^s$ and $U = V\cap H^{1/2-}$. Let $A = \phi_N^{-1}(U)$ and $A^M = (\phi_N^M)^{-1}(U)$. Since the images of $\phi_N$ and $\phi_N^M$ are almost surely included in $H^{1/2-}$, we have $A = \phi_N^{-1}(V)$ and $A^M = (\phi_N^M)^{-1}(V)$. The inequality \eqref{openset} can be rewritten as
$$
\mathbb P(A) \leq \liminf_{M\rightarrow \infty} \mathbb P(A^M) \; .
$$
Since $V$ is an open set, for all $\omega \in A$ there exists $\varepsilon > 0$ such that $\phi_N(\omega )+ B_\varepsilon $ is included in $V$. As $\phi_N^M$ converges towards $\phi_N$ in $L^2(\Omega,H^s)$, for almost all $\omega$, $\phi_N^M(\omega)$ converges towards $\phi_N(\omega)$. This is due to the fact that for all $\omega\in \Omega$, the sequence $\|\phi_N^M (\omega) -\phi_N(\omega)\|_{H^s}$ decreases when $M$ increases.

Hence, there exists $M_0$ such that for all $M\geq M_0$, $\phi_N^M(\omega) \in \phi_N(\omega) +B_\varepsilon \subset V$, thus $\omega$ belongs to $A^M$, that is to say
$$
\omega \in \bigcup_{M_0} \bigcap_{M\geq M_0} A^M = \liminf_{M\rightarrow \infty} A^M \;.
$$
Hence, $A$ is included in $\liminf A^M$ up to a set of $0$ probability. Therefore,
$$
\mathbb P(A) \leq \mathbb P(\liminf A^M)
$$
and the application of Fatou's lemma concludes the proof.
\end{proof}

We call $\mu_N^t$ the image measure of $\mu_N$ through $S(t)$, that is
$$
\mu_N^t (A) = \mu_N (S(t)^{-1}A)\; .
$$

\begin{proposition} For every $t$ and every $N$, we have that
$$
\mu_N^t = \mu_N \; .
$$
\end{proposition}

\begin{proof}Let $s<1/2$, $F$ be a closed set of $H^s$ and $B_\varepsilon$ be the open ball of radius $\varepsilon$ and center $0$ in $H^s$. As $F+B_\varepsilon$ is an open set of $H^s$ and $S(t)$ preserves the $H^s$ norm and thus the topology, we have that $S(t)^{-1}(F+B_\varepsilon)$ is an open set of $H^s$. Hence, we have that
$$
\mu_N^t( F) \leq \mu_N^t(F+B_\varepsilon) = \mu_N (S(t)^{-1}(F+B_\varepsilon))
$$
and using Proposition \ref{prop-openset} 
$$
\mu_N^t(F) \leq \liminf \mu_N^M(S(t)^{-1}(F+B_\varepsilon))\; .
$$
Since the vector space spanned by $\lbrace c_n,s_n, n=N+1 , \hdots, M\rbrace $ is stable under $S(t)$, that $S(t)$ is a rotation on this space and $\mu_N^M$ is a Gaussian vector supported on this space with mean value $0$, we have that $\mu_N^M$ is invariant under the transformation $S(t)$. Therefore, we get that
$$
\mu_N^M(S(t)^{-1}(F+B_\varepsilon)) = \mu_N^M(F+B_\varepsilon) \; .
$$
Then, we use that the $\liminf$ of a sequence is less than its $\limsup$ and that $B_\varepsilon$ is included in the closed ball $\overline{B_\varepsilon}$ of center $0$ and radius $\varepsilon$ to get
$$
\mu_N^t(F) \leq \limsup_{M\rightarrow \infty} \mu_N^M (F+\overline{B_\varepsilon})\; .
$$
Using Corollary \ref{cor-closet}, we get that
$$
\mu_N^t(F) \leq \mu_N(F+\overline{B_\varepsilon}) \; .
$$
Finally, we apply the dominated convergence theorem ($F$ is closed) to get
$$
\mu_N^t(F) \leq \mu_N(F)\; .
$$
To get the reverse inequality, we use the reversibility of the flow, yielding
$$
\mu_N(F) = \mu_N ((S(t)\circ S(-t))^{-1}F) = \mu_N^{-t}(S(t)^{-1}F)\leq \mu_N (S(t)^{-1}F) ) = \mu_N^t(F)
$$
Therefore, for any closed set of $H^s$,
$$
\mu_N^t(F) = \mu_N(F) \; ,
$$
and since this equality is preserved by taking the complementary sets and countable disjoint unions and is true for any closed set of $H^s$,$s<1/2$, it holds for all sets in the topological $\sigma$-algebra of $H^{1/2-}$.\end{proof}

\subsection{Invariance under the non linear flow}

\begin{proposition}The measure $\rho_N$ is invariant under the flow $\Psi_N(t)$ of the approximation of KdV 
\begin{equation}\label{approxkdv}
u_t + u_{xxx} + \Pi_N\left( \frac{(\Pi_N u)^2}{2}\right) = 0
\end{equation}
that is for every time $t$ and every measurable set $A$,
$$
\rho_N^t(A) : = \rho_N (\Psi_N(t)^{-1} A) = \rho_N(A) \; .
$$
\end{proposition}

\begin{proof} We recall that $\rho_N$ is defined as
$$
d\rho_N (u) = \kappa f_N(u) d\mu(u)
$$
with
$$
f_N(u) = \chi (\|u\|_{L^2}) e^{\frac{1}{6} (\Pi_N u)^3} = \chi (\|u\|_{L^2}) \chi (\|\Pi_N u\|_{L^2}) e^{\frac{1}{6} (\Pi_N u)^3}\; .
$$

We write $\rho_N$ as 
$$
d\rho_N(u) = \chi(\|u\|_{L^2})d\nu_N(\Pi_N u)\otimes d\mu_N((1-\Pi_N)u)
$$
with
$$
d\nu_N(\Pi_N u) = \chi(\|\Pi_N u\|_{L^2}) e^{\frac{1}{6}\int (\Pi_N u)^3} d\mu^N (u)\; .
$$
Recalling the structure of $\mu^N$, we have that
$$
d\nu_N(u = \sum_{n=1}^N a_n c_n + b_n s_n) = \chi( \Pi_N u) e^{\frac{1}{6}(\Pi_N u)^3 - \frac{1}{2}(\partial_x \Pi_N u)^2} \prod_{n=1}^N\frac{da_n db_n n^2}{2\pi}
$$
where
$$
\prod_{n=1}^N\frac{da_n db_n n^2}{2\pi}
$$
is the Lebesgue measure on $E_N$. The solution $\Psi_N(t) u$ can be written
$$
\Psi_N(t) u = \Psi_N(t)(\Pi_N u) + S(t) (1-\Pi_N) u
$$
and besides, the quantity
$$
\frac{1}{6}(\Pi_N u)^3 - \frac{1}{2}(\partial_x \Pi_N u)^2
$$
is an invariant of the equation \eqref{approxkdv} as well as the $L^2$ norm. As the equation \eqref{approxkdv} is Hamiltonian on $E_N$, the Lebesgue measure on $E_N$ is invariant through its flow thanks to Liouville theorem and the density of $\nu_N$ with regard to the Lebesgue measure is invariant as well, hence $\nu_N$ is invariant through $\Psi_N(t)$. Besides, $\mu_N$ is invariant through the flow $S(t)$. Finally, we have that for all $A_1 \in E_N$ and $A_2 \in E_N^\bot$, using the structure of $\Psi_N$ and $\nu_N$,
\begin{eqnarray*}
\nu_N \otimes \mu_N (\Psi_N(t)^{-1} (A_1\times A_2) ) &=& \nu_N \otimes \mu_N\Big(\Psi_N(t)^{-1}(A_1) \times S(t)^{-1}(A_2)\Big)\\
& = & \nu_N(\Psi_N(t)^{-1}A_1) \mu_N(S(t)^{-1}(A_2) 
\end{eqnarray*}
and then the invariance of $\nu_N$ under the flow $\Psi_N$ and of $\mu_N$ under $S(t)$,
\begin{eqnarray*}
\nu_N \otimes \mu_N (\Psi_N(t)^{-1} (A_1\times A_2) ) & = & \nu_N(A_1)\mu_N(A_2) \\
 & = & \nu_N\otimes \mu_N (A_1\times A_2) 
\end{eqnarray*}
where we see $A_1\times A_2$ as an isomorphic form of the set $\{ u \in H^{1/2-} \; : \; \Pi_N u \in A_1\; ,\; (1-\Pi_N)u \in A_2\}$.

Since the equality is true for every Cartesian product of $E_N$ and $E_N^\bot$, then it is true for all measurable set in the $\sigma$-algebra of the Cartesian product $E_0^N$ (finite dimensional) and $E_N$ (with topology $H^s$, $s<1/2$) that is for all measurable sets in $H^{1/2-}$.

Finally, as $\Psi_N(t)$ preserves the $L^2$ norm, $d\rho_N(u) = \chi(u) d\nu_N\otimes d\mu_N (u)$ is invariant through the flow of $\Psi_N(t)$.\end{proof}

\begin{proposition}The measure $\rho$ is invariant under the flow $\Psi(t)$ of KdV. \end{proposition}

\begin{proof} We recall from Lemma \ref{luc} that there exists a time $T>0$ such that for all $0\leq s_1<s_2<1/2$, all $R\geq 0$ and all $\varepsilon >0$, there exists $N_0$ such that for all $u\in B(s_2,R)\cap B(0,1)$ where $B(s,R)$ is the ball in $H^s$ of center $0$ and radius $R$, all $t\in[-T,T]$ and all $N\geq N_0$,
$$
\|\Psi(t)u-\Psi_N(t) u\|_{H^{s_1}} \leq \varepsilon\; .
$$
For $R\geq 0$ we call $T_n = nT $, $R_n = \sqrt n R$ and 
$$
A_n^N(s,R) = \lbrace \|\Psi_N(T_n)u\|_{H^s} \leq R_{n+1}\rbrace\; ,\; A^N(s,R) = \bigcap_n A_n^N(s,R) 
$$
$$
A(s,R) = \limsup_N A^N(s,R) \mbox{ and }A = \bigcup_{s,R}(A(s,R))\; .
$$
\begin{lemma}\label{lem-full} The set $A$ is of full $\rho$ measure.\end{lemma}

\begin{proof} We denote the complementary sets of any $A$ depending on various arguments by the letter $E$, for instance the complementary set of $A_n^N(s,R)$ in $H^{1/2-}$ is $E_n^N(s,R)$. First, 
$$
\rho_N(E_n^N(s,R))) = \rho_N (\|\Psi_N(T_n)\|_{H^s}> R_{n+1})
$$
and as $\rho_N$ is invariant through $\Psi_N(t)$, we have that
$$
\rho_N (E_n^N(s,R))  = \rho_N (\lbrace \|u\|_{H^s}> R_{n+1} \rbrace ) \leq C_se^{-c_s (n+1)R^2} \; .
$$
As $E^N(s,R) = \bigcup E_n^N(s,R)$, 
$$
\rho_N (E^N(s,R)) \leq \sum_{n\geq 0} \rho_N (E_n^N(s,R)) \leq C_s e^{-c_s R^2}
$$
therefore as for all measurable sets $\rho_N(A) \leq \rho(A) + d\|f-f_N\|_{L^1_\mu}$, we have that
$$
\rho(E^N(s,R)) \leq C_se^{-c_s R^2 } + d\|f-f_N\|_{L^1_\mu} \; .
$$
As $f_N$ converges towards $f$ in $L^1_\mu$ and thanks to Fatou's lemma, 
$$
\rho(E(s,R)) = \rho(\liminf_N E^N(s,R)) \leq \liminf_N \rho (E^N(s,R)) \leq C_se^{-c_s R^2}\; .
$$
Taking the intersection over $R$ gives
$$
\rho\left(\bigcap_R E(s,R)\right) = 0
$$
and the intersection over $s$ yields $\rho(E) = 0$. Therefore, $\rho(A) =1$.
\end{proof}

Let us fix $s_1<s_2$ and $R$ and prove by induction over $n$ that for all $t\in [-T_n,T_n]$, $\Psi_N(t) u$ converges towards $\Psi(t)u$ in $H^{s_1}$ uniformly for all $u\in A(s_2,R)$.

Initialization : $n=0$. As $T_0 = 0$ by definition, $\Psi_N(T_0)u = u=\Psi(T_0)u$.

$n\Rightarrow n+1$ : we assume that for all $t\in [-T_n,T_n]$, $\Psi_N(t)u$ converges towards $\Psi(t)u$ uniformly for $u$ in $A(s_2,R)$. Thanks to the fact that $u$ belongs to $A(s_2,R)$ we know that there is a subsequence $\Psi_{N_k}(T_n)u$ that satisfies 
$$
\|\Psi_{N_k}(T_n)u\|_{H^{s_2}}\leq R_{n+1}\; .
$$
Considering moreover the convergence of $\Psi_N(T_n) u$ in $H^{s_1}$, we get by duality that the $H^{s_2}$ norm of $\Psi(T_n)u$ is bounded by $R_{n+1}$. Then for all $t\in [0,T]$,
\begin{eqnarray}\label{inegtrig}
\|\Psi(T_n+t)u - \Psi_{N}(T_n+ t)u\|_{H^{s_1}} & \leq &\|\Psi(t) (\Psi(T_n) u) - \Psi_N(t)(\Psi(T_n) u)\|_{H^{s_1}} \\
&  & + \|\Psi_N(t) (\Psi(T_n) u) - \Psi_N(t) (\Psi_N(T_n)(u))\|_{H^{s_1}}\; .
\end{eqnarray}
As the $H^{s_2}$ norm of $\Psi(T_n) u$ is uniformly bounded in $u\in A(s_2,R)$, and that $\Psi_N(t)$ converges (as $t\leq T$) in $H^{s_1}$ towards $\Psi(t)$ uniformly on any bounded set of $H^{s_2}$, we have that 
$$
\|\Psi(t) (\Psi(T_n) u) - \Psi_N(t)(\Psi(T_n) u)\|_{H^{s_1}} 
$$
converges uniformly in $u\in A(s_2,R)$ towards $0$. 

Then, since $\Psi_N(T_n) u$ can be written as
$$
\Psi_N(T_n) (\Pi_N u) + S(t) ((1-\Pi_N)u)
$$
and $S(t)$ preserves the $H^s$ norms and that $\Psi_N(T_n)$ is continuous, as $\Pi_N A(s_2,R)$ is included in a compact set of $E_N$, we get that $\Psi_N(T_n) u$ is bounded uniformly in $u \in A(s_2,R)$ but not necessarily uniformly in $N$. The fact that it is bounded in $H^{s_1}$ uniformly in $N$ comes for the convergence of the sequence $\Psi_N(T_n)u$ towards $\Psi(T_n) u$ which is bounded by $R_{n+1}$ uniformly in $u$. Hence, there exists $R'$ such that for all $u\in A(s_2,R)$ and $N$, the $H^{s_1}$ norms of $\Psi(T_n)u$ and $\Psi_N(T_n) u$ are bounded by $R'$. As $\Psi_N(t)$ is Lipschitz on any bounded set with a constant independent from $N$, we get that
$$
\|\Psi_N(t) (\Psi(T_n) u) - \Psi_N(t) (\Psi_N(T_n)(u))\|_{H^{s_1}}
$$
converges towards $0$ uniformly in $u\in A(s_2,R)$. Indeed, we combine the uniform convergence of $\Psi_N(t) \Psi(T_n) u$ towards $\Psi(t) \Psi (T_n)u$ (for local reasons) and the uniform convergence of 
$$
\Psi_N(t)\Psi_N(T_n) u
$$ 
towards $\Psi_N(t) \Psi (T_n) u$ (using the induction hypothesis) in \eqref{inegtrig}. We get the uniform convergence of $\Psi_N(T_n+t)u $ towards $\Psi(T_n+t) u$ for all $t\in [-T,T]$ and by induction hypothesis, that $\Psi_N(t)u$ uniformly converges towards $\Psi(t) u$ for all $t\in [-T_n,T_{n+1}]$.

By using the same argument replacing $T_n$ by $-T_n$ and $t$ by $-t$, we get that for all $t$ in \\
$[-T_{n+1},T_{n+1}]$, the sequence $\Psi_N(t)u$ converges uniformly in $u \in A(s_2,R)$ in $H^{s_1}$.

We prove now the invariance of the measure. Let $t\in \R$ and $F$ be closed with regard to the topology $H^{s_1}$ and that its intersection with $H^{1/2-}$ is a subset of $A(s_2,R)$, then as $\Psi_N(t) $ converges uniformly in $A(s_2,R)$ towards $\Psi(t)$, we have that for all $\varepsilon >0$ there exists $N_0$ such that $\Psi(t)^{-1} (F)$ is almost surely included in $\Psi_N(t)^{-1}(F+B_\varepsilon)$, for all $N\geq N_0$ hence
$$
\rho^t(F) \leq \rho (\Psi_N(t)^{-1} (F+B_\varepsilon))\;.
$$
Then, comparing $\rho$ and $\rho_N$, we have
$$
\rho (\Psi_N(t)^{-1} (F+B_\varepsilon)) \leq \rho_N (\Psi_N(t)^{-1}(F+B_\varepsilon) ) + \kappa\|f-f_N\|_{L^1_\mu}\; .
$$
Using the invariance of $\rho_N$ under $\Psi_N$, we have that
$$
\rho^t(F) \leq \rho_N(F+B_\varepsilon) + \kappa\|f-f_N\|_{L^1_\mu} \leq \rho(F+B_\varepsilon) +2 \kappa\|f-f_N\|_{L^1_\mu} \; .
$$
We let $N$ go to $\infty$ such that
$$
\rho^t(F) \leq \rho (F+B_\varepsilon)
$$
and by applying the dominated convergence theorem,
$$
\rho^t (F) \leq \rho(F) \; .
$$
The reverse inequality comes from the continuity and the reversibility of the flow.

To get this inequality for all closed subset of $H^s$, we need to prove the following lemma.
\begin{lemma}\label{lem-adh}Let $s_1 < s_3 < s_2$. The closure of $A(s_2,R)$ in $H^{s_1}$ is included in $\widetilde A(s_3,2R,n)$ for all $n$ with
$$
\widetilde A(s_3,2R,n) =  \limsup_N \bigcap_{k=0}^n A_k^N(s_3,2R)\; .
$$
\end{lemma}

\begin{remark} The previous proof allows us to say that $\rho^t(F)$ is equal to $\rho(F)$ for all time $t\in [-T_n,T_n]$ as long as $F$ is a closed subset of $H^{s_1}$ whose intersection with $H^{1/2-}$ is included in $A(s_2,R,n)$.\end{remark}

\begin{proof} of Lemma \ref{lem-adh} Let $u_j$ be a sequence of $A(s_2,R)$ that converges towards $u$ in $H^{s_1}$. As $u_j$ is in $A(s_2,R)$, for all $k$, $\Psi (T_k)u_j$ is bounded by $R_{k+1}$. But since $\Psi_N(T_k) $ converges uniformly in $A(s_2,R)$ in any $H^s$ with $s< s_2$, there exists $N_0(k)$ such that for all $N\geq N_0(k)$, $\|\Psi_{N}(T_k)(u_j)\|_{H^{s_3}}$ is bounded by $2R_{k+1}$. Since $u_j$ converges towards $u$ and $\Psi_N(T_k)$ is continuous, $\Psi_N(T_k)u_j$ converges towards $\Psi_N(T_k)u$ in $H^{s_1}$, but thanks to the uniform bound in $j$ of $\|\Psi_N(T_k)u_j\|_{H^{s_2}}$, $\|\Psi_N(T_k)u_j\|_{H^{s_2}}$ is bounded by $2R_{k+1}$ as long as $N\geq N_0(k)$. By considering $N_0 = \max N_0(k)$ for $k=0,\hdots ,n$, we get that $u$ belongs to 
$$
\bigcap_{N\geq N_0} \bigcap_{k=0}^n A_k^N(s_3,R) \subseteq \bigcup_{M_0} \bigcap_{N\geq M_0} \bigcap_{k=0}^n A_k^N(s_3,R) = \liminf \bigcap_{k=0}^n A_k^N(s,R)
$$
which is included in the $\limsup$, that is $\widetilde A(s_3,2R,n)$.\end{proof}

We now fix $t$ and $F$ a closed set of $H^{s_1}$. As the sequence $T_n$ goes to $\infty$, $t$ belongs to $[-T_n,T_n]$ for some $n$. For $R\geq 0$, we have that, with $\overline{A(s_2,R)}^c$ the complementary set of the adherence of $A(s_2,R)$ in $H^{s_1}$
$$
\rho^t (\overline{A(s_2,R)}^c) = 1 - \rho^t(\overline{A(s_2,R)})
$$
and since $\overline{A(s_2,R})$ is a closed set in $H^{s_1}$ included in $\widetilde A(s_3,2R, n)$,
$$
\rho^t(\overline{A(s_2,R)}) = \rho(\overline{A(s_2,R)}) \geq \rho (A(s_2,R))
$$
hence
$$
\rho^t (\overline{A(s_2,R)}^c) \leq C_{s_2}e^{-c_{s_2} R^2}\; .
$$
Therefore, for all $R$ ,
\begin{eqnarray*}
|\rho^t(F) - \rho(F)| &\leq & |\rho^t(F) - \rho^t(F\cap \overline{A(s_2,R)})| +  \\
 & & |\rho^t(F\cap \overline{A(s_2,R)}-\rho(F\cap \overline{A(s_2,R)}|+ |\rho(F) - \rho(F\cap \overline{A(s_2,R)})|
\end{eqnarray*}
as $F\cap \overline{A(s_2,R)}$ is a closed set of $H^{s_1}$ included in $\widetilde A(s_3,2R,n)$,  $\rho^t(F\cap \overline{A(s_2,R)}-\rho(F\cap \overline{A(s_2,R)}$ is equal to $0$. Then,
\begin{eqnarray*}
|\rho^t(F) - \rho(F)| & \leq & \rho^t (\overline{A(s_2,R)}^c)  + \rho (\overline{A(s_2,R)}^c)  \\
 & \leq & C_{s_2} e^{-c_{s_2}R^2}
\end{eqnarray*}
We let $R$ go to $\infty$ and we get that for all closed set of $H^s$,
$$
\rho^t(F) = \rho(F)\; .
$$

Since this equality is preserved by taking the complementary sets and by countable disjoint unions and that it is true for all $s_1$, we get that for all measurable sets $M$ of $H^{1/2-}$, $\rho^t(M) = \rho(M)$.\end{proof}

\begin{proposition} The invariant measure $\rho$ belongs to $M_{s,p}$ for all $s<1/2$ and $p<\infty$. \end{proposition}

\begin{proof}First, since $\rho$ is defined by
$$
d\rho (u) = d\chi(\|u\|_{L^2}) e^{\frac{1}{6} \int u^3}d\mu(u) \; ,
$$
we have that $\|u\|_{L^2(\T)}$ is less than $1$ $\rho$-almost surely, thus $u$ belongs to $L^\infty_\rho,L^2(\T)$.

Then, we need to prove that $u$ belongs to $L^p_\rho, H^s(\T)$. This is equivalent to the fact that $\|u\|_{H^s}^p$ is $\rho$ integrable, that is, that $\|u\|_{H^s}^p \chi(\|u\|_{L^2}) e^{\frac{1}{6}\int u^3}$ is $\mu$ integrable. But we know that $\|u\|_{H^s}$ satisfies $\mu$-large Gaussian deviation estimates 
$$
\mu(\|u\|_{H^s} \geq R ) \leq C_s e^{-c_s R^2}
$$
which ensures that $\|u\|_{H^s}^q$ is $\mu$-integrable for all $q$, and we recall
$$
\chi(\|u\|_{L^2}) e^{\frac{1}{6}\int u^3} \leq e^{\|u\|_{L^\infty}}
$$
together with the large Gaussian deviation estimates of $\|u\|_{L^\infty}$, that is
$$
\mu( \|u\|_{L^\infty} \geq R) \leq Ce^{-c R^2} \; ,
$$
which ensures that $\left( \chi(\|u\|_{L^2})e^{\frac{1}{6}\int u^3} \right)^q$ is $\mu$-integrable for all $q$. Finally, we get that
$$
\|u\|_{H^s}^p \chi(\|u\|_{L^2}) e^{\frac{1}{6}\int u^3}
$$
is integrable with regard to $\mu$. 

Therefore, $\rho$ belongs to $M_{s,p}$. \end{proof}

\begin{proposition} The measure $\rho$ is stable in $M_{s,p}$ for the flow of KdV, in the sense that for all times $t$, there exist two constants $C,c$ such that for all $\nu\in M_{s,p}$,
$$
\|\nu^t - \nu\|_{s,p} \leq Ce^{c|t|(1+\|x\|_{L^\infty_\nu,L^2})^{12}}(1+\|x\|_{L^p_\nu,H^s})\|\nu - \rho\|_{s,p} \; .
$$
\end{proposition}

\begin{proof} We use the invariance of $\rho$ to write
$$
\|\nu^t - \nu\|_{s,p} \leq \|\nu^t - \rho^t\|_{s,p} + \|\rho - \nu\|_{s,p}
$$
and the continuity of the flow of KdV for the Wasserstein metrics to get
$$
 \|\nu^t - \rho^t\|_{s,p} \leq Ce^{c|t|(1+\|x\|_{L^\infty_\nu,L^2})^{12}}(1+\|x\|_{L^p_\nu,H^s}) \|\nu - \rho\| \; .
$$\end{proof}

\bibliographystyle{amsplain}
\bibliography{bibannso} 
\nocite{*}

\end{document}